\documentclass[12pt]{amsart}
\usepackage{pgfplots}
\usepackage[utf8]{inputenc}
\usepackage[colorlinks=true]{hyperref}
\usepackage{amsmath,amscd,amsbsy,amssymb,amsthm,amsfonts}
\usepackage{latexsym,url,bm}
\usepackage{cite,enumitem}
\usepackage{xcolor}
\usepackage{tikz}

\newcommand{\cL}{{\mathcal L}}

\newcommand{\jbx}{\langle \alpha\rangle}

\usepackage{fullpage}

\newtheorem{theorem}{Theorem}[section]
\newtheorem{lemma}[theorem]{Lemma}
\newtheorem{proposition}[theorem]{Proposition}
\newtheorem{corollary}[theorem]{Corollary}

\theoremstyle{remark}
\newtheorem{remark}{Remark}
\numberwithin{equation}{section}
\numberwithin{figure}{section}

\title{Two dimensional solitary water waves with constant vorticity, Part I: the deep gravity case }

\author{James Rowan}
\address{Department of Mathematics, University of California at Berkeley}
\curraddr{}
\email{jrowan@math.berkeley.edu}

\author{Lizhe Wan}
\address{Department of Mathematics, University of Wisconsin - Madison}
\curraddr{}
\email{lwan33@wisc.edu}

\pgfplotsset{compat=1.17}

\begin{document}
\maketitle

\begin{abstract}
  We consider the two dimensional pure gravity water waves with nonzero constant vorticity in infinite depth, working in the holomorphic coordinates introduced in~\cite{hunter2016two}.
  We show that close to the critical velocity corresponding to zero frequency, a solitary wave exists.
  We use a fixed point argument to construct the solitary wave whose profile resembles a rescaled Benjamin-Ono soliton.
  The solitary wave is smooth and has an asymptotic expansion in terms of powers of the Benjamin-Ono soliton.
\end{abstract}

\section{Introduction}

We consider two dimension gravity water waves with nonzero constant vorticity and infinite depth, but without surface tension or viscosity.
The fluid occupies a time dependent domain $\Omega (t) \subset \mathbb{R}^2$ with infinite depth and a free upper boundary $\Gamma(t)$ which is asymptotically approaching $y=0$.
Denoting the fluid velocity by $\mathbf{u}(t,x,y) = (u(t,x,y), v(t,x,y))$, the pressure by $p(t,x,y)$, and the constant vorticity by $\gamma \neq 0$, the equations inside $\Omega_t$ are
\begin{equation*}
\left\{
             \begin{array}{lr}
            u_t +uu_x +vu_y = -p_x &  \\
            v_t + uv_x +vv_y = -p_y -g& \\
            u_x +v_y =0 & \\
            \omega = u_y -v_x = -\gamma.
             \end{array}
\right.
\end{equation*}
On the boundary $\Gamma_t$ we have the dynamic boundary condition
\begin{equation*}
    p =0,
\end{equation*}
and the kinematic boundary condition
\begin{equation*}
    \partial_t +\mathbf{u}\cdot \nabla \text{ is tangent to }\Gamma_t.
\end{equation*}
Here $g>0$ is the gravitational constant. 

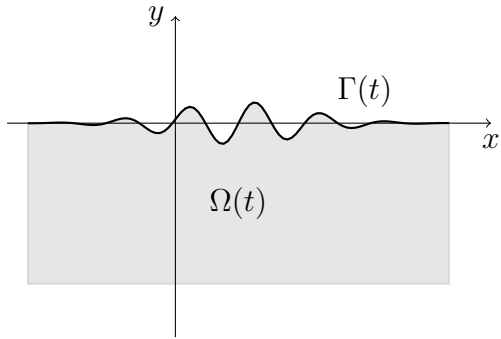
\begin{figure}
  \centering

  \begin{tikzpicture}
    \begin{axis}[ xmin=-12, xmax=12.5, ymin=-2.5, ymax=1.5, axis x
      line = none, axis y line = none, samples=100 ]

      \addplot+[mark=none,domain=-10:10,stack plots=y]
      {0.2*sin(deg(2*x))*exp(-x^2/20)};
      \addplot+[mark=none,fill=gray!20!white,draw=gray!30,thick,domain=-10:10,stack
      plots=y] {-1.5-0.2*sin(deg(2*x))*exp(-x^2/20)} \closedcycle;
      \addplot+[black, thick,mark=none,domain=-10:10,stack plots=y]
      {1.5+0.2*sin(deg(2*x))*exp(-x^2/20)}; 
      %\addplot+[black,mark=none,domain=-10:10,stack plots=y]
      %{-1.5-0.2*sin(deg(2*x))*exp(-x^2/20)};

      \draw[->] (axis cs:-3,-2) -- (axis cs:-3,1) node[left] {\(y\)};
      \draw[->] (axis cs:-11,0) -- (axis cs:12,0) node[below] {\(x\)};
      \filldraw (axis cs:-3,-1.5)  node[above left]
      {\(\)}; \node at (axis cs:0,-0.75) {\(\Omega(t)\)}; \node at
      (axis cs:6,0.3) {\(\Gamma(t)\)};

    \end{axis}
  \end{tikzpicture}
  \caption{The fluid domain.}
\end{figure}

In \cite{MR4462478} and\cite{MR3869381}, the above system was expressed in holomorphic position/velocity potential coordinates $(W,Q)$, and the water wave equations with constant vorticity have the following form:
\begin{equation}
\left\{
             \begin{array}{lr}
             W_t + (W_\alpha +1)\underline{F} +i\dfrac{\gamma}{2}W = 0 &  \\
             Q_t - igW +\underline{F}Q_\alpha +i\gamma Q +\mathbf{P}\left[\dfrac{|Q_\alpha|^2}{J}\right]- i\dfrac{\gamma}{2}T_1 =0,&  
             \end{array}
\right.\label{e:CVWW}
\end{equation}
where $J := |1+ W_\alpha|^2$ and $\mathbf{P}$ is the projection onto negative frequencies, namely
\begin{equation*}
    \mathbf{P} = \frac{1}{2}(\mathbf{I} - iH),
\end{equation*}
with $H$ denoting the Hilbert transform, and
\begin{equation*}
\begin{aligned}
&F: = \mathbf{P}\left[\frac{Q_\alpha - \Bar{Q}_\alpha}{J}\right], \quad &F_1 = \mathbf{P}\left[\frac{W}{1+\Bar{W}_\alpha}+\frac{\Bar{W}}{1+W_\alpha}\right],\\
&\underline{F}: =F- i \frac{\gamma}{2}F_1,  \quad &T_1: = \mathbf{P}\left[\frac{W\Bar{Q}_\alpha}{1+\Bar{W}_\alpha}-\frac{\Bar{W}Q_\alpha}{1+W_\alpha}\right].
\end{aligned}
\end{equation*}

We seek to find solitary wave solutions of~\eqref{e:CVWW}, in other words nonzero solutions having the form 
\begin{equation*}
    \left(W(\alpha, t), Q(\alpha, t)\right) = \left(W(\alpha + ct), Q(\alpha + ct)\right)
\end{equation*}
for $t\geq 0$.
Here we assume that $\lim_{\alpha\to\infty}W(\alpha)=0$, so that we only consider localized solitary waves and not periodic traveling waves, but both the solitary and periodic traveling waves problems have been widely studied.

The presence or absence of physical parameters such as gravity, surface tension, (infinite) depth, and vorticity can lead to or prevent the formation of solitary waves.
Beginning with the pioneering results in \cite{MR65317, MR445136,MR629699}, solitary waves are known to exist in finite depth for both two and higher dimensions. 
Gravity and gravity-capillary water waves have shown on a vast number of literature, including \cite{MR963906, MR1720395, MR1378603, MR2867413, MR1133302, MR2263898, MR2073504, MR2069635, MR4412559,MR4246394, MR3780138, MR3461361}.
As for the existence of solitary waves for water waves with constant vorticity, most results are on the gravity-capillary case, especially in finite depth, see for instance \cite{MR3961580, MR3415532, MR1364730, MR1297863}.

\par The mathematical theory of periodic travelling waves is rich and well-developed for different types of waves.
Steady periodic water waves with vorticity were considered in for instance \cite{MR2027299,MR2835865, MR2969824, MR2754336}.
It was shown that for  steady periodic gravity water waves with vorticity, steady periodic surface water waves with vorticity, and steady periodic deep water waves with vorticity, the solutions are symmetric, see \cite{MR2362244,MR2256915,MR2144685, MR2329142}.
Furthermore, for rotational travelling water waves, the solution is even analytic, see \cite{MR2754336,MR2753609,MR2902476}.

A variety of different coordinate systems have been used to study the existence of solitary waves; we mention two other works that also use holomorphic coordinates.
In the case of zero vorticity $\gamma = 0$ and infinite depth, Ifrim and Tataru showed in \cite{MR4151415} that there are no nontrivial solitary waves.
Then using a similar method in \cite{MR4455193}, Ifrim-Pineau-Tataru-Taylor showed that  no pure capillary solitary waves exist in finite depth.

The authors, together with Ifrim and Tataru showed in \cite{MR4462478} that for solutions concentrated at the amplitude and frequency scale $\epsilon$,  a rescaling of the Benjamin-Ono equation
\begin{equation}
    (\partial_t+\partial_x|D|)u + uu_x =0 \label{BenjaminOno}
\end{equation}
gives a good and stable approximation to the system \eqref{e:CVWW} on the cubic time scale $\left[0,T\epsilon^{-2}\right]$ where well-posedness was proved in~\cite{MR3869381}.
Here $|D|$ is a Fourier multiplier defined in \eqref{DDefinition}.

The approximation can be heuristically justified on the linear level by looking at the linearization of~\eqref{e:CVWW} around the zero solution,
\begin{equation*}
\left\{
             \begin{array}{lr}
             w_t + q_\alpha = 0 &  \\
             q_t - igw  +i\gamma q =0,&  
             \end{array}
\right.\label{e:CVWWlinearized}
\end{equation*}
as well as its dispersion relation
\begin{equation*}
    \tau^2+\gamma\tau+g\xi=0.
\end{equation*}
Taking a quadratic approximation for the top branch of the dispersion relation (recall that $\xi\le 0$ in the holomorphic coordinate setting), we have
\begin{equation*}
    \tau=-\frac{g}{\gamma}\xi-\frac{g^2}{\gamma^3}\xi^2,
\end{equation*}
so that $-\frac{g}{\gamma}$ corresponds to $0$ frequency.
The velocity $-\frac g \gamma$ will be the critical velocity in the sequel.

It is well known that the Benjamin-Ono equation has the unique single 
soliton solutions
\begin{equation}
    u (t,c) =c\rho(c(x-ct)), \quad \mbox{ where } \rho(x)=\frac{4}{1+x^2}, \label{e:rhodefinition}
\end{equation}
and $c$ is the velocity of the waves in the horizontal direction.
By the result in \cite{MR4462478}, the system \eqref{e:CVWW} has solitary-wave-like solutions on a cubic time scale $[0,T\epsilon^{-2}]$.
A natural question then arises, are there exact solitary waves solutions of \eqref{e:CVWW} close to (after a suitable series of rescalings) the Benjamin-Ono soliton?
We give an affirmative answer to this question.

Instead of working with the solitary wave ansatz in \eqref{e:CVWW} in the complex setting,
we derive and then work with a real-valued differential equation, the \emph{Babenko equation}, for $U: = \Im W$:
\begin{align}
    (g+ c\gamma-c^2|D|)U = - \frac{\gamma^2}{2}U^2-gU|D|U -\frac{g}{2}|D|U^2 +\frac{\gamma^2}{2}\left(U|D|U^2 - U^2|D|U - \frac{1}{3}|D|U^3\right). \label{e:BabenkoEqnLR}
\end{align}
Historically,  Babenko studied periodic traveling waves~\cite{MR899856, MR898306}, but the approach of finding an elliptic equation solved by the solitary wave profile works in the nonperiodic case as well.
The existence of a solitary wave for \eqref{e:CVWW} is equivalent to the existence of a nontrivial solution of the equation \eqref{e:BabenkoEqnLR}, as we will show in Section~\ref{s:Babenko}.

\par Our first result is that under some suitable conditions, the equation \eqref{e:BabenkoEqnLR} is an elliptic equation, so that the solutions enjoy
good Sobolev regularity. 
\begin{theorem}
Assume that \eqref{e:BabenkoEqnLR} has a solution $\Im W\in H^2$. If $g+c\gamma <0$ and $2g \sup\Im W < c^2$, the solution of \eqref{e:BabenkoEqnLR} is in $\cap_{k=2}^\infty H^k$. \label{t:TheoremTwo}
\end{theorem}

Note in particular that the sign of the vorticity $\gamma$ determines the sign of the possible velocities $c$.
If $\gamma>0$, then the velocity $c<-\frac g \gamma$, and if $\gamma<0$, then $c>-\frac{g}{\gamma}$.
In both cases, the solitary wave velociy $c$ is just outside the range of velocities for dispersive waves.

\par Theorem \ref{t:TheoremTwo} motivates us to find a nontrivial
$H^2$ solution of \eqref{e:BabenkoEqnLR} in order to gain better Sobolev regularity.
Our main result shows that such a nontrivial solution does exist for certain velocities:
\begin{theorem} 
Given any nonzero vorticity $\gamma$ and positive constant $g$, there exists a small positive constant $\epsilon_0$ such that for any velocity $c$ satisfying $-\epsilon_0 <g+c\gamma <0$, the system \eqref{e:CVWW} has a unique nontrivial $H^2$ solitary waves solution in the sense that $\Im W$ is close to the Benjamin-Ono soliton after rescaling and space translation; its imaginary part is an even function, and its real part is an odd function.  \label{t:TheoremOno}  
\end{theorem}

Here the even and odd functions should be regarded as the class of functions modulo space translations.
For example, we can define the $H^2_e(\mathbb{R})$, the $H^2$ space of even functions on $\mathbb{R}$ as
\begin{equation*}
    H^2_e(\mathbb{R}): = \{f\in H^2(\mathbb{R}): f(\alpha) = f(2\alpha_0 -\alpha) \text{  a.e. for some constant } \alpha_0\in\mathbb{R}\}.
\end{equation*}
Without loss of generality, we can simply let $\alpha_0 = 0$.

We will show in Section \ref{s:slowexist} that $\Im W$ and $\Im Q$ are even functions.
A key observation is that the Hilbert transform turns the even functions into odd functions.
Since $\Re W = H\Im W$, and $\Re Q = H\Im Q$,  we know that $\Re W$ and $\Re Q$ are both odd functions.
By the discussion of Section $6$ of \cite{MR4436142}, we know the the profile of the water waves surface $\eta$ in Zakharov-Craig-Sulem formulation is proportional to $\Im W$ in the holomorphic coordinates.
This shows that the profile of the solitary wave surface is an even function modulo space translations.

\begin{remark}
In a recent paper \cite{Lokharu22Amplitude}, Lokharu, Wahlén and Weber showed that for two-dimensional steady pure-gravity water waves with finite depth of positive constant vorticity $\gamma$,   $\sup \Im W - \inf \Im W <\frac{2g}{\gamma^2}$. 
Since $0$ is in the range of $\Im W$,
we have $\sup\Im W<\frac{2g}{\gamma^2}$, and $\inf \Im W > -\frac{2g}{\gamma^2}$.
Their result does not contradict our Theorem \ref{t:TheoremTwo}, since in our setting, $\left|c+\frac{g}{\gamma}\right|$ is small and positive, so that $\frac{2g}{\gamma^2}<\frac{2c^2}{g}$, which gives an upper bound on $\Im W$.
\end{remark}

\begin{remark}
We consider in this article the simple case where the solitary waves have a simple profile with velocity $c$.
However,  the simplified model Benjamin-Ono equation  \eqref{BenjaminOno} itself has $N$-soliton solutions,  with $N$ different velocities. 
It is natural to ask whether the system \eqref{e:CVWW} also has multi-solitary waves solutions, but this is beyond the scope of the present article here.
\end{remark}

Furthermore, we show that the profile of the solution can be expressed as a power series in the Benjamin-Ono soliton $\rho$ \eqref{e:rhodefinition}.

\begin{theorem}
Let $\epsilon : = \left|\frac{g}{\gamma}+c\right|$, then the profile of water waves surface $\Im W$ modulo translations has an asymptotic expansion in terms of $\rho$ of the form, $\forall N \geq 1$,
   \begin{equation}
        \Im W(\alpha)=\sum_{k=1}^{N} b_k\rho^k\left(\frac{\epsilon\alpha}{(1+\epsilon)^2}\right)+g_N,\label{e:ClaimedFormOne}
    \end{equation} 
where  the coefficients $b_k$ depend on $\epsilon$ and the remainder term $g_N$ decays faster than $\jbx^{-2N}$; 
in particular, $\Im W(\alpha)= O_{L^\infty}\left((1+\alpha^2)^{-1}\right)$ asymptotically. \label{TheoremThree}
\end{theorem}

A more precise functional framework and statement for this theorem will be given in Section~\ref{s:property}.

Finally we prove the following theorem, giving the characterization of $\Im W$ at or near the maximal speed:

\begin{theorem}\label{t:maximal}
    Let $\left(c_M,-\frac g \gamma\right)$ for $\gamma>0$ (or $\left(-\frac g \gamma,c_M\right)$ for $\gamma<0$) be the maximal interval of velocities for which a continuous one-parameter family of even solitary wave solutions to~\eqref{e:CVWW} extending the solutions of Theorem~\ref{t:TheoremOno} exists in a weighted Sobolev space $X$.
    Then one of the following must happen:
    \begin{enumerate}
        \item the interval of possible velocities is infinite, so that $c_M=-\infty$ for $\gamma>0$ (or $c_M=\infty$ for $\gamma<0$); 
        
        \item the norm of the solitary wave profiles in a weighted Sobolev space $X$ is increasing to infinity along a subsequence as $c$ approaches $c_M$; 

        \item the solitary wave is approaching the maximal height along a subsequence as $c$ approaches $c_M$;  

        \item an eigenvalue (other than the one corresponding to translation symmetry) for the linearization of~\eqref{e:BabenkoEqnLR} around the solitary wave is approaching $0$ along a subsequence as $c$ approaches $c_M$.
    \end{enumerate}
\end{theorem}

The article is organized as follows. In  Section \ref{s:Babenko} we use the variational approach to derive the Babenko equation \eqref{e:BabenkoEqnLR}, which is a nonlinear first order ODE that describes the profile of the solitary waves of \eqref{e:CVWW}. 
For the following sections, we shall be studying the properties of the Babenko equation \eqref{e:BabenkoEqnLR}.
In Section \ref{s:ellipticity}, we rewrite the Babenko equation \eqref{e:BabenkoEqnLR}, and show that under the \textit{maximal height condition}, the equation is a nonlinear first order elliptic equation.
With this ellipticity condition, we are able to show the solution which we construct in the following section have better Sobolev regularity.
In Section \ref{s:slowexist}, we rescale the Babenko equation, and use the contraction principle to prove the existence of the rescaled equation in the even functional spaces.
In Section \ref{s:property}, we consider the asymptotic behavior of the solitary waves solution and prove it has an $\jbx^{-2}$ spacial decay.
We then give a finer asymptotic decay of the solution profile and shows that it only involve the even powers of $\jbx^{-1}$, which is the result of Theorem \ref{TheoremThree}.
Finally in Section \ref{s:perturbation}, we briefly discuss continuation of the construction for velocities near one where a solitary wave exists, and give conditions in terms of that solitary wave and the linearized Babenko equation around it under which the range of velocities can be extended.

\textbf{Acknowledgments.} The authors would like to thank their Ph.D. advisors, Daniel Tataru and Mihaela Ifrim, for many helpful discussions and suggestions during the preparation of this paper.
The first author was partially supported by the NSF grant DMS-2054975 as well as the Simons Foundation (via Daniel Tataru's Simons Investigator Grant).

\section{Derivation of the Babenko equation} \label{s:Babenko}
 In this section, we derive the Babenko equation of the water waves system \eqref{e:CVWW}, which describes the profile of solitary waves.
 The system \eqref{e:CVWW} is expressed in holomorphic coordinates, and $W, Q$ are both holomorphic functions.
A very nice property of holomorphic functions is  that their real parts equal the Hilbert transform of their imaginary parts.
Therefore it suffices to consider just the imaginary parts.
According to the computation carried out in the appendix of \cite{MR3869381},
the total energy of \eqref{e:CVWW} is
\begin{equation}
    \mathcal{E} = \frac{1}{2}\int |D|\Im Q \cdot\Im Q +g(\Im W)^2(1+ |D|\Im W) + \gamma |D|\Im Q \cdot (\Im W)^2 + \frac{\gamma^2}{3}(\Im W)^3(1+|D|\Im W)\, d\alpha, \label{TotalEnergy}
\end{equation}
and the horizontal momentum is
\begin{equation}
    \mathcal{P} = -\int |D|\Im Q \Im W + \frac{\gamma}{2}(\Im W)^2(1+ |D|\Im W) \, d\alpha . \label{Momentum}
\end{equation}
Here, let $PV_{.}$ denotes principle value, the differential operator $|D|$ is defined by
\begin{equation*}
    |D|f(\alpha) = \partial_\alpha Hf(\alpha)  = PV_{.} \frac{1}{\pi}\int \frac{f(\alpha)-f(y)}{|\alpha-y|^2}\,dy.
\end{equation*}
An alternative definition of operator $|D|$ which we will use frequently is given via the Fourier transform
\begin{equation}
    \widehat{|D|f}(\xi) = |\xi|\hat{f}(\xi).\label{DDefinition}
\end{equation}
We derive the Babenko equation with the variational approach.
In \cite{MR688749}, it was observed by Benjamin and Olver that the solution of Babenko equation is characterised as a critical point of the total energy subject to the constraint of fixed momentum.
The velocity $c$ can be viewed as the corresponding Lagrange multiplier.
In order to compute the functional derivative, we will need the following result.
\begin{lemma}
   Given a function $f(x)$ defined on $\mathbb{R}^n$, let $L\left(f(x), \nabla f(x), |D|f(x)\right)$ to be a $C^2_0$ functional of $f$, $\nabla f(x)$, and $|D|f(x)$, and consider the functional defined by
   \begin{equation*}
       J(f) = \int L\left(f(x), \nabla f(x), |D|f(x)\right) dx.
   \end{equation*}
   Then the functional derivative of $J(f)$ is given by
   \begin{equation*}
       \frac{\delta J}{\delta f} = \frac{\partial L}{\partial f}- div \left(\frac{\partial L}{\partial \nabla f}\right)+|D|\left(\frac{\partial L}{\partial |D|f}\right).
   \end{equation*}
\end{lemma}
\begin{proof}
    For any smooth test function $\phi(x)$ and small constant $\epsilon$, we have
    \begin{align*}
        \delta J(f,\epsilon \phi):=& J(f+\epsilon\phi) - J(f) \\
        =& \int L\left(f(x)+\epsilon\phi(x), \nabla f(x)+\epsilon\nabla \phi(x), |D|f(x)+\epsilon|D|\phi(x)\right) -L\left(f(x), \nabla f(x), |D|f(x)\right)\,dx \\
        =& \int \frac{\partial L}{\partial f}\epsilon \phi + \frac{\partial L}{\partial \nabla f} \epsilon \nabla\phi + \frac{\partial f}{\partial |D|f}\epsilon |D|\phi\, dx +O(\epsilon^2)\\
        =& \epsilon \int \left(\frac{\partial L}{\partial f}- div\left(\frac{\partial L}{\partial \nabla f}\right)+|D|\left(\frac{\partial L}{\partial |D|f}\right)\right) \phi\, dx+O(\epsilon^2),
    \end{align*}
using integration by parts and the self-adjoint property of operator $|D|$.
The result follows from the definition
    \begin{equation*}
        \int \frac{\delta J}{\delta f} \phi(x) dx = \frac{\delta J(f,\epsilon \phi)}{\epsilon}\Big |_{\epsilon = 0}.
    \end{equation*}
\end{proof}
Considering the functional derivative with respect to $\Im Q$, the first equation is
\begin{equation*}
\frac{\delta \mathcal{E}}{\delta \Im Q} = c \frac{\delta \mathcal{P}}{\delta \Im Q},
\end{equation*}
we have
\begin{equation}
\Im Q = -\frac{\gamma}{2}(\Im W)^2 - c\Im W. \label{e:BabenkoQ}
\end{equation}

Similarly, for the functional derivative with respect to $\Im W$, the second equation is
\begin{equation*}
\frac{\delta \mathcal{E}}{\delta \Im W} = c \frac{\delta \mathcal{P}}{\delta \Im W}.
\end{equation*}
We have by computation
\begin{align*}
  \frac{\delta \mathcal{E}}{\delta \Im W} =& g\Im W(1+|D|\Im W) +\frac{1}{2}g|D|(\Im W)^2 +\gamma |D|\Im Q \Im W + \frac{\gamma^2}{2}(\Im W)^2(1+|D|\Im W) + \frac{\gamma^2}{6}|D|(\Im W)^3,\\
  \frac{\delta \mathcal{P}}{\delta \Im W} =& -|D|\Im Q -\gamma \Im W(1+|D|\Im W)-\frac{\gamma}{2}|D|(\Im W)^2,
\end{align*}
which leads to
\begin{align*}
    g\Im W(1+|D|\Im W) +\frac{1}{2}g|D|(\Im W)^2 +\gamma |D|\Im Q \Im W + \frac{\gamma^2}{2}(\Im W)^2(1+|D|\Im W) \nonumber \\
    + \frac{\gamma^2}{6}|D|(\Im W)^3 +c|D|\Im Q +c\gamma \Im W(1+|D|\Im W)+ c\frac{\gamma}{2}|D|(\Im W)^2 =0. 
\end{align*}

Therefore by eliminating $\Im Q$ using \eqref{e:BabenkoQ}, we obtain that $U:=\Im W$ solves the  Babenko equation
\eqref{e:BabenkoEqnLR}.

\par As a consequence of the derivation, we immediately obtain
\begin{corollary}
If a nonzero $U\in H^2_0$  solves \eqref{e:BabenkoEqnLR}, then \eqref{e:CVWW} has a solitary wave solution modulo space translations
\begin{equation*}
    \left( HU(\alpha + ct) + iU(\alpha + ct),   cHU(\alpha + ct)-\frac{\gamma}{2}H(U^2)(\alpha + ct)+icU(\alpha + ct)-\frac{i\gamma}{2}U^2(\alpha + ct)\right).
\end{equation*}
\end{corollary}

\section{Elliptic regularity} \label{s:ellipticity}

In this section, we show that the Babenko equation \eqref{e:BabenkoEqnLR} is essentially an elliptic differential equation under some suitable conditions, so that any $H^2$ solution is a-priori smooth. 
The construction of such a non-trivial $H^2$ solution is performed in Section \ref{s:slowexist}.
\par Assume for the moment for $s\geq 2$, we have found a nontrivial $H^s$ solution of the Babenko equation \eqref{e:BabenkoEqnLR}.
We rewrite \eqref{e:BabenkoEqnLR} as
\begin{align}
    (g+c\gamma &-c^2|D|+ 2gT_{U} |D|)U  = -\frac{\gamma^2}{2}U^2 -\frac{1}{2}g|D|\Pi(U, U) - \frac{\gamma^2}{2}|D|\Pi(U,U) U \nonumber\\
    &   - \frac{\gamma^2}{6}|D|\Pi(U,U,U)  -g[T_{U}, |D|]U  +\gamma^2U[T_{U}, |D|]U - \frac{\gamma^2}{2}[T_{U^2}, |D|]U \nonumber\\
     & -gT_{|D|U}U - g\Pi(U, |D|U) -\frac{\gamma^2}{2}T_{|D|U}U^2 -\frac{\gamma^2}{2}\Pi(|D|U, U^2) \nonumber\\
     & + \gamma^2 T_{T_{U}|D|U}U + \gamma^2 \Pi(|D|U, U^2) +\gamma^2(T_{U^2}-T_{U}T_{U})|D|U. \label{e:Paradifferential}
\end{align}
Here, $T_u v$ is the low-high frequency component of paraproduct, and $\Pi(u,v)$ is the high-high frequency component of paraproduct, please see for instance Chapter $2$ of \cite{MR2768550} for definition and further introduction.
\par We define the elliptic operator $L$ by
\begin{equation*}
    LU:= (g+c\gamma -c^2|D|+ 2gT_{U} |D|)U.
\end{equation*}

By writing the Babenko equation \eqref{e:BabenkoEqnLR} as \eqref{e:Paradifferential}, each cubic term on the right hand side has either cancellation or can be estimated in $H^s$ norm as indicated below.

We now estimate the right hand side of \eqref{e:Paradifferential}.
The first quadratic term can be bounded by $\| U\|_{H^s}^2$.
For the next three terms of the right hand side of \eqref{e:Paradifferential}, we use the following lemma on paraproduct.
\begin{lemma}[high-high paraproduct] \label{l:highHigh}
 Let $\alpha, \beta \in \mathbb{R}$. If $\alpha + \beta >0$, then 
 \begin{equation*}
     \|\Pi(u,v)\|_{H^{\alpha+\beta - \frac{d}{2}}(\mathbb{R}^d)} \lesssim \| u\|_{H^{\alpha}(\mathbb{R}^d)}\| v\|_{H^{\beta}(\mathbb{R}^d)}.
 \end{equation*}
\end{lemma}
Then for $s\geq \frac{3}{2}$,
\begin{align*}
    \|g|D|\Pi(U, U) \|_{H^s}&\lesssim \|\Pi(U, U)\|_{H^{s+1}} \lesssim \|U\|_{H^s}^2,\\
    \||D|\Pi(U,U) U \|_{H^s}&\lesssim \|U\|_{H^s}\|\Pi(U, U) \|_{H^{s+1}} \lesssim \|U\|_{H^s}^3,\\
    \||D|\Pi(U,U,U) \|_{H^s}&\lesssim \|\Pi(U,U,U) \|_{H^{s+1}} \lesssim  \|U\|_{H^s}^3.
\end{align*}
Then we estimate three commutator terms.
\begin{lemma}
\begin{enumerate}
    \item  For any $s\in \mathbb{R}$,  $u\in C^1(\mathbb{R})$ and $f\in H^s(\mathbb{R})$, 
    \begin{equation*}
        \| [T_u, |D|]f\|_{H^s} \lesssim \| u\|_{C^1}\|f\|_{H^s}.
    \end{equation*}
    \item For a function $u(x)\in L^\infty$, then $T_u$ is an operator of order 0, and  $\forall s\in \mathbb{R}$,
    \begin{equation*}
        \| T_u f\|_{H^s} \lesssim \|u\|_{L^\infty}\|f\|_{H^s}.
    \end{equation*}
\end{enumerate}
\end{lemma}
For $s>\frac{3}{2}$, $H^s(\mathbb{R})$ can be embedded into $L^\infty(\mathbb{R})$, then we have
\begin{align*}
    \|g[T_{U}, |D|]U \|_{H^s} &\lesssim \|U\|_{H^s}^2, \\
    \|\gamma^2U[T_{U}, |D|]U \|_{H^s}&\lesssim \|U\|_{H^s}\| [T_{U}, |D|]U\|\lesssim \|U\|_{H^s}^3, \\
    \| \frac{\gamma^2}{2}[T_{U^2}, |D|]U\|_{H^s} & \lesssim \|U^2\|_{H^s}\|U\|_{H^s}\lesssim \|U\|_{H^s}^3.
\end{align*}

We now estimate the last seven terms of \eqref{e:Paradifferential}.
For $s>\frac{3}{2}$, 
\begin{align*}
    \|gT_{|D|U}U \|_{H^s} &\lesssim \||D|U \|_{L^\infty}\|U \|_{H^s}\lesssim\| U\|_{H^s}^2,\\
    \|g \Pi(U, |D|U)\|_{H^s} &\lesssim \||D|U \|_{H^{s-1}}\|U \|_{H^s}\lesssim\| U\|_{H^s}^2,\\
    \| -\frac{\gamma^2}{2}T_{|D|U} U^2\|_{H^s}&\lesssim \||D|U \|_{L^\infty}\|U^2 \|_{H^s}\lesssim\| U\|_{H^s}^3,\\
    \| \frac{\gamma^2}{2}\Pi(|D|U, U^2)\|_{H^s}&\lesssim \||D|U \|_{H^{s-1}}\|U^2 \|_{H^s}\lesssim\| U\|_{H^s}^3,\\
    \|\gamma^2 T_{T_{U}|D|U}U \|_{H^s}&\lesssim \|T_{U}|D|U \|_{L^\infty}\|U \|_{H^s}\lesssim\| U\|_{H^s}^3,\\
    \| \gamma^2 \Pi(|D|U, U^2)\|_{H^s}&\lesssim \||D|U \|_{H^{s-1}}\|U^2 \|_{H^s}\lesssim\| U\|_{H^s}^3,\\
    \|\gamma^2(T_{U^2}-T_{U}T_{U})|D|U \|_{H^s}&\lesssim \|U\|_{H^s}^3.
\end{align*}
Furthermore, we want to replace the $T_{U}$ by $U$ in $L$, to do this , we need the following lemma.
\begin{lemma}[\cite{MR3260858} Proposition 2.10] \label{l:paraproductDifference}
   If $s_0, s_1, s_2\in \mathbb{R}$ satisfies  $s_0< s_1 + s_2 -\frac{d}{2}$ and $s_0 \leq s_1$, then for all $a\in H^{s_1}(\mathbb{R}^d)$ and $u\in H^{s_2}(\mathbb{R}^d)$,
 \begin{equation*}
     \|au -T_au\|_{H^{s_0}(\mathbb{R}^d)} \lesssim \| a\|_{H^{s_1}(\mathbb{R}^d)}\| u\|_{H^{s_2}(\mathbb{R}^d)}.
 \end{equation*}  
\end{lemma}
Then we have
\begin{equation*}
    \|(T_{U}-U)|D|U\|_{H^s}\lesssim \| U\|_{H^s}\||D|U\|_{H^{s-1}} \lesssim   \| U\|_{H^s}^2.
\end{equation*}
Define 
\begin{equation*}
    \tilde{L}u  = (g+c\gamma -c^2|D|+ 2gU |D|)u,
\end{equation*}
and $f$ equals the right hand side of \eqref{e:Paradifferential} plus $(T_{U}-U)|D|U$
We have thus shown that $ \tilde{L}U  = f\in H^s(\mathbb{R})$.
\par To gain more regularity of the solution, we need to assume the following sign condition and the (possible) maximal height condition:
\begin{equation}
    g+c\gamma<0, \label{SignCondition}
\end{equation}
\begin{equation}
    2gU < c^2. \label{MaximalHeight}
\end{equation}
 \begin{proposition} \label{p:maximumheight}
    If the sign condition \eqref{SignCondition}
    and the maximal height condition \eqref{MaximalHeight}
    both hold, then the operator $\tilde{L}$ is elliptic of order $1$.
\end{proposition}

For elliptic equations, we recall the classical lemma of elliptic regularity.
\begin{lemma}[Elliptic regularity]
 For an  elliptic operator $\Phi$ of order $m$, if $u, \Phi u \in H^s$, then $u\in H^{s+m}$, and 
 \begin{equation*}
     \|u\|_{H^{s+m}}\lesssim \|u\|_{H^{s}} + \|\Phi u\|_{H^{s}}. \label{t:Elliptic}
 \end{equation*} 
\end{lemma}
Given that $U\in H^s(\mathbb{R})$, above classical theory of elliptic regularity shows that $U \in H^{s+1}(\mathbb{R})$.
Iterating this process gives $U \in H^N(\mathbb{R})$, for any $N>0$.

\section{Existence of nontrivial solution to the Babenko equation}\label{s:slowexist}

In this section, we use the Benjamin-Ono soliton` to construct solitary waves for~\eqref{e:CVWW} with velocities slightly above the critical velocity $-\frac{g}{\gamma}$.

We note that the constant vorticity water waves problem admits a two-parameter family of scaling symmetries, and that we have a sign reversal symmetry where we can change the signs of both $\gamma$ and $c$.
This means we can rescale the equations so that $g=1$ and $\gamma=-1$.
After this rescaling, the assumption that the velocity is close to the critical velocity means we can write
\begin{equation*}
    c= 1+\epsilon
\end{equation*}
for $0<\epsilon<\epsilon_0$.

The rescaled version of the Babenko equation~\eqref{e:BabenkoEqnLR}, we have  
\begin{equation*}
    -\epsilon U-(1+\epsilon)^2|D|U +\frac{1}{2}U^2+U|D|U +\frac{1}{2}|D|U^2
    -\frac{1}{2}U|D|U^2  + \frac{1}{2}U^2|D|U + \frac{1}{6}|D|U^3=0.
\end{equation*}

We now rescale again, using the scaling 
\begin{equation}
    U (\alpha) = \epsilon \phi\left(\frac{\epsilon \alpha}{(1+\epsilon)^2}\right) \label{UPhiRelation}
\end{equation}
and dividing by $\epsilon^2$, treating the terms with higher powers of $\epsilon$ as perturbative:
\begin{equation}
    \begin{split}
        -\phi-|D|\phi +\frac{1}{2}\phi^2&= \frac{\epsilon}{(1+\epsilon)^2}\left(-\phi|D|\phi-\frac{1}{2}|D|\phi^2\right)\\
        &\quad +\frac{\epsilon^2}{(1+\epsilon)^2}\left(\frac{1}{2}\phi|D|\phi^2-\frac{1}{2}\phi^2|D|\phi-\frac{1}{6}|D|\phi^3\right).
    \end{split}\label{e:babenkoeqnrescaled}
\end{equation}

Now consider \begin{equation*}
    \rho(\alpha)=\frac{4}{1+\alpha^2}.
\end{equation*}
This is the single soliton solution of velocity $1$ for the Benjamin-Ono equation; it satisfies the equation
\begin{equation*}
    -\rho-|D|\rho+\frac{1}{2}\rho^2=0.
\end{equation*}

When $\epsilon$ is small, we will construct solutions to~\eqref{e:babenkoeqnrescaled} close to $\rho$.
Subtracting the equations for $\phi$ and $\rho$, we find that $v=\phi-\rho$ solves the following equation:
\begin{equation}
    -v-|D|v+\rho v=-\frac{1}{2}v^2+\tilde g_\epsilon(v+\rho).\label{e:veqn}
\end{equation}
where $\tilde g_\epsilon(v+\rho)$ is the right-hand side of~\eqref{e:babenkoeqnrescaled}

Let 
\begin{equation*}
    \mathcal L=-1-|D| +\rho
\end{equation*}
denote the differential operator on the left-hand side of \eqref{e:veqn}.
This operator arises in the study of the linearization of the Benjamin-Ono equation around its soliton solution $\rho$.
Its principal part is elliptic, and it sends functions in $H^1$ to $L^2$.

The spectrum of the operator $\mathcal L$ is well-understood, see for example ~\cite{bennett1983solitary} or ~\cite{kenig2009stability}.
There are four eigenvalues \begin{equation*}
    \left\{-1,\frac{1-\sqrt 5}{2},0,\frac{1+\sqrt 5}{2}\right\},
\end{equation*} along with a continuous spectrum $(-\infty,-1]$, and the eigenvalue $0$ is simple with $\rho'$ as the corresponding eigenfunction.
The eigenfunction corresponding to the eigenvalue  $0$ is also an odd function which can be computed explicitly.

 Since the kernel of the operator $\mathcal{L}$ is spanned by $\rho^{'}$, where $\rho^{'}$ is an odd function, we work in spaces based on $H^2_e(\mathbb R)$, the space of all even functions in $H^2(\mathbb R)$.
Note that $H^2_e(\mathbb R)$ is an algebra, and that $|D|$ maps $H^k_e(\mathbb R)$ into the space $H^{k-1}_e(\mathbb R)$ of even functions in $H^{k-1}(\mathbb R)$.
In this setting, the inverse of $\mathcal{L}$ can be uniquely determined.

We write,
\begin{equation*}
    P_\epsilon(v)=\mathcal L^{-1}\left(-\frac 1 2 v^2+\tilde g_\epsilon(v+\rho)\right).
\end{equation*}
If we can find a fixed point of this operator, then we have a solution to~\eqref{e:veqn}.

Define the weighted Sobolev space $X$ to be the space of all even functions such that
\begin{equation}
    \|u\|^2_{X}=\|u\|^2_{H^2}+\eta^2\|\alpha u\|^2_{H^2}\label{e:fixedpointspace}
\end{equation}
is finite, where $\eta$ is a small constant to be chosen later.

We will show that $P_\epsilon$ is a contraction on a small ball $B(0,\delta)$ in $X$, so that it has a unique fixed point in that ball.
We work in the weighted space $X$ rather than merely $H^2$ to get better control on the asymptotic $\alpha$ behavior of $v$ and thus a clearer picture of the shape of the solitary wave.

We have \begin{equation*}
    P_\epsilon(u)-P_\epsilon(v)=\mathcal L^{-1}\left(-\frac 1 2 (u-v)(u+v)+\tilde g_\epsilon(u+\rho)-\tilde g_\epsilon(v+\rho)\right).
\end{equation*}

We first show that 
\begin{equation}
        \left\|-\frac 1 2 (u-v)(u+v)+\tilde g_\epsilon(u+\rho)-\tilde g_\epsilon(v+\rho)\right\|_{H^1}\le C(\delta,\epsilon) \|u-v\|_{H^2}.\label{e:innerfixedpointbound}
\end{equation}
Recall the definition 
\begin{equation*}
    \begin{split}
        \tilde g_\epsilon(\phi)&=\frac{\epsilon}{(1+\epsilon)^2}\left(-\phi|D|\phi-\frac{1}{2}|D|\phi^2\right)\\
        &\quad +\frac{\epsilon^2}{(1+\epsilon)^2}\left(\frac{1}{2}\phi|D|\phi^2-\frac{1}{2}\phi^2|D|\phi-\frac{1}{6}|D|\phi^3\right).
    \end{split}
\end{equation*} 
In investigating the difference
\begin{equation*}
    \tilde g_\epsilon(u+\rho)-\tilde g_\epsilon(v+\rho),
\end{equation*} we work with each term individually, also tracking how many powers of $\rho$ appear.

\textbf{Terms not in $\tilde g_\epsilon(u+\rho)-\tilde g_\epsilon(v+\rho)$}: We have 
\begin{equation*}
    \left\|-\frac 1 2 (u-v)(u+v)\right\|_{H^1}\le \delta \|u-v\|_{H^2}.
\end{equation*}

\textbf{Second-order terms}: The second-order terms without any factors of $\rho$ contribute
\begin{equation*}
    \frac{\epsilon}{(1+\epsilon)^2}\left\|-u|D|u+v|D|v\right\|_{H^1}=\frac{\epsilon}{(1+\epsilon)^2}\|u|D|(u-v)+(u-v)|D|v\|_{H^1}\le 2\frac{\epsilon}{(1+\epsilon)^2}\delta\|u-v\|_{H^2}
\end{equation*}
 and
\begin{equation*}
    \frac{\epsilon}{(1+\epsilon)^2}\left\|-\frac{1}{2}|D|u^2+\frac{1}{2}|D|v^2\right\|_{H^1}=\frac{\epsilon}{2(1+\epsilon)^2}\||D|(u^2-v^2)\|_{ H^1}\le \frac{\epsilon}{(1+\epsilon)^2}\delta \|u-v\|_{H^2}.
\end{equation*}
The second-order terms with one factor of $\rho$ contribute
\begin{align*}
    \frac{\epsilon}{(1+\epsilon)^2}\left\|-u|D|\rho+v|D|\rho\right\|_{H^1}&\le \frac{\epsilon}{(1+\epsilon)^2}\|\rho\|_{ H^2}\|u-v\|_{H^2},\\
    \frac{\epsilon}{(1+\epsilon)^2}\left\|-\rho|D|u+\rho|D|v\right\|_{H^1}&\le \frac{\epsilon}{(1+\epsilon)^2}\|\rho\|_{L^\infty}\|u-v\|_{H^2},\text{ and}\\
    \frac{\epsilon}{(1+\epsilon)^2}\left\|-|D|(\rho u)+|D|(\rho v)\right\|_{H^1}&\le \frac{\epsilon}{(1+\epsilon)^2}\|\rho\|_{H^2}\|u-v\|_{H^2}.
\end{align*}

\textbf{Third-order terms}: The third-order terms with no factors of $\rho$ contribute
\begin{align*}
    \frac{\epsilon^2}{(1+\epsilon)^2}\left\|\frac{1}{2}u|D|u^2-\frac{1}{2}v|D|v^2\right\|_{H^1}&\le \frac{\epsilon^2}{2(1+\epsilon)^2}\left\|(u-v)|D|u^2+v|D|(u^2-v^2)\right\|_{H^1}\nonumber\\
    &\le\frac{3\epsilon^2\delta^2}{2(1+\epsilon)^2}\|u-v\|_{H^2},\\
    \frac{\epsilon^2}{(1+\epsilon)^2}\left\|-\frac{1}{2}u^2|D|u+\frac{1}{2}v^2|D|v\right\|_{H^1}&\le \frac{\epsilon^2}{2(1+\epsilon)^2}\left\|(u^2-v^2)|D|u+v^2|D|(u-v)\right\|_{H^1}\nonumber\\
    &\le\frac{3\epsilon^2\delta^2}{2(1+\epsilon)^2}\|u-v\|_{H^2}, \text{ and}\\
        \frac{\epsilon^2}{(1+\epsilon)^2}\left\|-\frac{1}{6}|D|u^3+\frac{1}{6}|D|v^3\right\|_{H^1}&\le \frac{\epsilon^2}{6(1+\epsilon)^2}\left\|u^3-v^3\right\|_{H^2}\le\frac{\epsilon^2\delta^2}{2(1+\epsilon)^2}\|u-v\|_{H^2}.
\end{align*}
The third-order terms with one factor of $\rho$ contribute
\begin{align*}
    \frac{\epsilon^2}{(1+\epsilon)^2}\left\|\frac{1}{2}\rho|D|u^2-\frac{1}{2}\rho|D|v^2 \right\|_{H^1}&\le \frac{\epsilon^2}{(1+\epsilon)^2}\delta \|\rho\|_{H^1}\|u-v\|_{H^2},\\
    \frac{\epsilon^2}{(1+\epsilon)^2}\left\|u|D|(u\rho)-v|D|(v\rho) \right\|_{H^1}&\le \frac{\epsilon^2}{(1+\epsilon)^2}\delta \|\rho\|_{H^2}\|u-v\|_{H^2},\\
    \frac{\epsilon^2}{(1+\epsilon)^2}\left\|-(u\rho)|D|u+(v\rho)|D|v \right\|_{H^1}&\le \frac{\epsilon^2}{(1+\epsilon)^2}\delta \|\rho\|_{H^2}\|u-v\|_{H^2},\\
    \frac{\epsilon^2}{(1+\epsilon)^2}\left\|-\frac{1}{2}u^2|D|\rho+\frac{1}{2}v^2|D|\rho \right\|_{H^1}&\le \frac{\epsilon^2}{(1+\epsilon)^2}\delta \|\rho\|_{H^2}\|u-v\|_{H^2},\text{ and}\\
    \frac{\epsilon^2}{(1+\epsilon)^2}\left\|-\frac{1}{2}|D|(u^2\rho)+\frac{1}{2}|D|(v^2\rho)\right\|_{H^1}&\le \frac{\epsilon^2}{(1+\epsilon)^2}\delta \|\rho\|_{H^2}\|u-v\|_{H^2}.
\end{align*}
The third-order terms with two factors of $\rho$ contribute
\begin{align*}
    \frac{\epsilon^2}{(1+\epsilon)^2}\left\|\rho|D|(u\rho)-\rho|D|(v\rho) \right\|_{H^1}&\le  \frac{\epsilon^2}{(1+\epsilon)^2}\|\rho\|_{H^2}^2\|u-v\|_{H^2},\\
    \frac{\epsilon^2}{(1+\epsilon)^2}\left\|\frac{1}{2}u|D|\rho^2-\frac{1}{2}v|D|\rho^2 \right\|_{H^1}&\le \frac{\epsilon^2}{2(1+\epsilon)^2} \|\rho\|_{H^2}^2\|u-v\|_{H^2},\\
    \frac{\epsilon^2}{(1+\epsilon)^2}\left\|-\frac{1}{2}\rho^2|D|u+\frac{\epsilon^2}{2(1+\epsilon)^2}\rho^2|D|v \right\|_{H^1}&\le \frac{\epsilon^2}{2(1+\epsilon)^2} \|\rho\|_{H^2}^2\|u-v\|_{H^2},\\
    \frac{\epsilon^2}{(1+\epsilon)^2}\left\|-(u\rho)|D|\rho+(v\rho)|D|\rho \right\|_{H^1}&\le \frac{\epsilon^2}{(1+\epsilon)^2}\|\rho\|_{H^1} \|\rho\|_{ H^2}\|u-v\|_{H^2},\text{ and}\\
    \frac{\epsilon^2}{(1+\epsilon)^2}\left\|-\frac{1}{2}|D|(u\rho^2)+\frac{1}{2}|D|(v\rho^2)\right\|_{H^1}&\le \frac{\epsilon^2}{2(1+\epsilon)^2} \|\rho\|_{H^2}^2\|u-v\|_{H^2}.%\label{e:lastestimatefixedpoint}
\end{align*}

Combining above estimates, we have the desired bound~\eqref{e:innerfixedpointbound}.

Because the operator $\mathcal L$, when restricted to even functions, has two nonzero eigenvalues in its discrete spectrum 
\begin{equation*}
    \frac{1+\sqrt 5}{2}\text{ and }\frac{1-\sqrt 5}{2}
\end{equation*} 
a continuous spectrum contained in $(-\infty,-1]$, we know that
\begin{equation}
    \|\mathcal L^{-1}g\|_{L^2}\le \left(\frac{\sqrt 5-1}{2}\right)^{-1}\|g\|_{L^2}.\label{e:eigenvalueboundforLinv}
\end{equation}

Since the operator $1+|D|$ is elliptic, we know that
\begin{equation}
    \|g\|_{H^1}\lesssim \|g\|_{L^2}+\|(1+|D|) g\|_{L^2},\label{e:ellipticRegularityStart}
\end{equation}
where the implicit constant does not depend on $g$.
Moreover, since $\rho$ is bounded, this elliptic regularity statement tells us that if 
\begin{equation*}
    (1+|D|)f=\rho f-g,
\end{equation*}
then
\begin{equation*}
    \|f\|_{H^1}\lesssim \|f\|_{L^2}+\|\rho f-g\|_{L^2}\lesssim \|f\|_{L^2}+\|g\|_{L^2},
\end{equation*}
or, in other words, that
\begin{equation}
    \|\mathcal L^{-1}g\|_{H^1}\lesssim\|\mathcal L^{-1}g\|_{L^2}+\|g\|_{L^2}.\label{e:initialEllipticRegularity}
\end{equation}
By applying the spectral bound~\eqref{e:eigenvalueboundforLinv}, we have
\begin{equation*}
    \|\mathcal L^{-1}g\|_{L^2}\lesssim \|g\|_{L^2},
\end{equation*}
so that~\eqref{e:initialEllipticRegularity} becomes
\begin{equation*}
    \|\mathcal L^{-1}g\|_{H^1}\lesssim \|g\|_{L^2}.
\end{equation*}
Running the same elliptic regularity argument again, we have \begin{equation}
    \|\mathcal L^{-1}g\|_{H^2}\lesssim\|\mathcal L^{-1}g\|_{H^1}+\|g\|_{H^1}.\label{e:ellipticRegularityEnd}
\end{equation}
Applying this to 
\begin{equation*}
    P_\epsilon(u)-P_\epsilon(v)=\mathcal L^{-1}\left(-\frac 1 2 (u-v)(u+v)+\tilde g_\epsilon(u+\rho)-\tilde g_\epsilon(v+\rho)\right)
\end{equation*}
and using~\eqref{e:innerfixedpointbound}, we have
\begin{equation}
    \left\|P_{\epsilon}(u)-P_{\epsilon}(v)\right\|_{H^2}\le C(\delta,\epsilon) \|u-v\|_{H^2},\label{e:fixedpointboundwithconstants}
\end{equation}
so we have a contraction on $H^2_e$ as long as $\delta$ and $\epsilon$ are small enough.

We now turn to the weighted space $(\eta\alpha)^{-1} H^2_e$.
Multiplying~\eqref{e:veqn} by $\eta\alpha$, we have
\begin{equation*}
    \eta\alpha v-|D|(\eta\alpha v)+\rho (\eta\alpha v)=\eta[\alpha,|D|]v-\frac 1 2 v(\eta\alpha v)+\eta\alpha \tilde g_{\epsilon}(v+\rho).\label{e:veqnweighted}
\end{equation*}

Writing $w=\eta\alpha v$, then
\begin{equation}
    \mathcal L w=\eta[\alpha,|D|]v-\frac 1 2 vw+\eta\alpha \tilde g_{\epsilon}(v+\rho).\label{e:weqn}
\end{equation}

Let
\begin{equation*}
    Q_{\epsilon}(w)=\mathcal L^{-1}\left(\eta\eta\alpha,|D|]\left(\frac{1}{\eta\alpha} w\right)-\frac 1 2 \left(\frac{w}{\eta\alpha}\right)w+\eta\alpha \tilde g_{\epsilon}\left(\frac{w}{\eta\alpha}+\rho\right)\right).
\end{equation*}
If we can show that $Q_\epsilon$ is a contraction from $X$ onto $H_e^2$,
then together with the fact that $P_\epsilon$ is a contraction on $H^2_e$,
we can show that $P_\epsilon$ is a contraction on $X$.

We now show, for $w_u=\eta\alpha u$, $w_v=\eta\alpha v$ solving~\eqref{e:weqn},
\begin{equation}
        \left\|\eta[\alpha,|D|](u-v)-\frac{1}{2}\eta\alpha\left(u^2-v^2\right)+\eta\alpha\left(\tilde g_\epsilon(u+\rho)-\tilde g_\epsilon(v+\rho)\right)\right\|_{H^1}\le C \| (u-v)\|_{X}\label{e:innerfixedpointboundweighted}
\end{equation}
for small enough values of $\eta$, where the constant $C$ depends on $\eta$, $\delta$ and $\epsilon$.

\textbf{Commutator terms:}
A computation shows that, for $f\in H^2$ that vanishes at infinity,
\begin{equation*}
    \left[\alpha,|D|\right]f=\alpha H(f_\alpha)-H(\partial_\alpha(\alpha f))=-Hf+[\alpha,H]f_\alpha=-Hf+2\int f_\alpha d\alpha =-Hf,
\end{equation*} 
where we work on the Fourier side to conclude that (in the sense of distributions) \begin{equation*}
    \mathcal F [\alpha,H]=[\frac 1 i \partial_\xi, i\operatorname{sgn} \xi]=2\delta_0,
\end{equation*} 
so that \begin{equation*}
[\alpha,H]f_\alpha=\int 1\cdot f_\alpha d\alpha.
\end{equation*}
giving us the estimate
\begin{equation}
    \left\|[\alpha,|D|]f\right\|_{H^2}\le \|f\|_{H^2},\label{e:commutatorweightboundintermediate}
\end{equation}
so that, as long as $\eta$ is chosen small enough, we have 
\begin{equation*}
    \left\|\eta [\alpha,|D|](u-v)\right\|_{H^2}\le \frac{1}{4} \|u-v\|_{H^2}.
\end{equation*}

\textbf{Remaining terms}: the remaining terms can be handled as in the previous unweighted argument, with the $\eta\alpha$ in front placed with whichever difference term is measured $H^1$ or $H^2$.
This proves the rest of~\eqref{e:innerfixedpointboundweighted}.

Arguing with elliptic rgularity as before, we have
\begin{equation*}
    \left\|Q_\epsilon(w_u)-Q_{\epsilon}(w_v)\right\|_{H^2}\le C(\delta,\epsilon)\|w_u-w_v\|_{H^2}.
\end{equation*}
As long as $\delta$ and $\epsilon$ are chosen small enough to ensure the contribution of the non-commutator terms is at least as small as that of the commutator terms, we have that $Q_\epsilon$ is a contraction on $H^2_e$,
and thus $P_\epsilon$ is a contraction on $X$.

We now need to choose a $\delta=\delta(\epsilon)$ such that $P_\epsilon$ sends the ball of radius $\delta$ in $X$ into itself.

Set $\delta=k\epsilon$ for some positive constant $k$ to be chosen later.

We have, arguing as above, 
\begin{equation*}
    \|P_\epsilon(u)\|_{H^2}\lesssim\left(\frac 1 2 \|u^2\|_{H^2}+\|\tilde g_\epsilon(u+\rho)\|_{H^2}\right).
\end{equation*}

For the $u^2$ term, we have
\begin{equation*}
    \|u^2\|_{H^2}\le \|u\|_{H^2}^2<k^2\epsilon^2.
\end{equation*} 

For the terms with no factors of $u$, we have 
\begin{equation*}
    \|\tilde g_\epsilon(\rho)\|_{H^2}=C\epsilon+O_{H^2}(\epsilon^2)
\end{equation*} 
for some positive constant $C$.

For all other terms, there will be at least one factor of $u$.
This factor can be bounded in $H^1$, $L^\infty$, or $H^2$ by a factor of $k\epsilon$, so, coupled with the factors of $\epsilon$ in the definition of $\tilde g_\epsilon$, we can guarantee that 
\begin{equation*}
        \|\tilde g_\epsilon(\rho+u)\|_{H^2}=C\epsilon+O_{H^2}(\epsilon^2).
\end{equation*}

We now work in the weighted space $(\eta\alpha)^{-1}H^2$.
We have
\begin{equation*}
    \|Q_\epsilon(\eta\alpha u)\|_{H^2}\lesssim \left(\frac{1}{2\eta} \|u\alpha u\|_{H^2}+\|\eta[\alpha,|D|]u\|_{H^2}+\|\eta\alpha \tilde g_\epsilon(u+\rho)\|_{H^2}\right).
\end{equation*}
Since $\alpha\rho$ and $\eta\alpha u$ are bounded in $H^2$, we can argue as above.
The only change is the addition of the commutator term which can be handled using the estimate~\eqref{e:commutatorweightboundintermediate}.
We thus have
\begin{equation*}
    \|Q_\epsilon(\eta\alpha u)\|\le C'\epsilon+O_{H^2}(\epsilon^2),
\end{equation*}
so by choosing $k$ to be a sufficiently large numerical constant, we can ensure that for small enough $\epsilon$, 
\begin{equation*}
    \|P_\epsilon(u)\|_{H^2}+\|Q_\epsilon(\eta\alpha u)\|_{H^2}<k\epsilon=\delta
\end{equation*} 
for $u$ in a ball of radius $\delta$ around the origin in $X$.

Plugging this choice of $\delta$ into the right-hand sides of above related inequalities, we see that by picking $\epsilon_0$ small enough, we can guarantee that the constant on the right-hand side of~\eqref{e:fixedpointboundwithconstants} is smaller than $1$ for all $0<\epsilon<\epsilon_0$. 
This means that we have found a $\delta$ such that $P_\epsilon$ is a contraction on the ball of radius $k\epsilon$ around the origin in $X$ for all $\epsilon<\epsilon_0$.

After untangling the definitions and rescalings, this guarantees the existence of a solitary wave of velocity 
\begin{equation*}
    c=-(1+\epsilon)\frac{g}{\gamma}
\end{equation*} 
for the constant vorticity water wave problem for a small open interval $\epsilon\in(0,\epsilon_0)$.

\section{Asymptotic decay of the profile} \label{s:property}

In this section, we give a detailed analysis of the asymptotic decay of the function $\phi$ constructed in Section \ref{s:slowexist}.

For $\epsilon$ small enough, from \eqref{UPhiRelation} we see that the maximal height $\|U\|_{L^\infty}=O(\epsilon)$ is also small.
The velocity is close to $-\frac{g}{\gamma}$, so that $\frac{c^2}{2g}\approx \frac{g}{2\gamma^2}\gg \epsilon$.
The maximal height condition \eqref{MaximalHeight} holds in this case.
By the result in Section \ref{s:ellipticity}, we conclude that $U\in H^N(\mathbb{R})$, for any $N>0$.
Hence, we should also expect $\phi$ to have nice regularity and decay properties when $\epsilon$ is small.

Before going to a more refined asymptotic analysis, we first give a qualitative bound of asymptotic decay for $\phi$. 
We recall the equation \eqref{e:veqn}, which we write as
\begin{equation*}
    -(1+|D|)v=\mathcal N(v).\label{e:vFromNofv}
\end{equation*}
We have
\begin{equation}
    \mathcal N(v)=-\rho v-\frac 1 2 v^2+\frac{\epsilon}{(1+\epsilon)^2}\mathcal N_1(v+\rho)+\frac{\epsilon^2}{(1+\epsilon)^2}\mathcal N_2(v+\rho),\label{e:nonlinearitygeneral}
\end{equation}
where
\begin{equation}
    \mathcal N_1(\phi)=-\phi|D|\phi-\frac 1 2 |D|\phi^2,\qquad \mathcal N_2(\phi)=\left(\frac 1 2 \phi|D|\phi^2-\frac 1 2 \phi^2 |D|\phi -\frac 1 6 |D|\phi^3\right).\label{e:nonlinearityexplicitdecay}
\end{equation}

According to the result in \cite{MR1106251}, the operator $1+|D|$ is an elliptic operator of order $1$ with fundamental solution
\begin{equation*}
    \Phi(\alpha)=\frac 1 \pi \int_0^\infty \frac{\tau e^{-\tau}}{\alpha^2+\tau^2}\, d\tau.
\end{equation*}
$\Phi(\alpha)$  is a positive function that satisfies the properties
\begin{equation*}
    \int_{-\infty}^\infty \Phi(\alpha) d\alpha = 1,\quad \Phi(\alpha) \lesssim \frac{1}{\alpha^2}.
\end{equation*}
Using Lemma \ref{l:structureofmodDrho} and Lemma \ref{l:structureofmodDf} below, a direct computation show that $|\mathcal{N}(v)|\lesssim \frac{\epsilon}{(1+\epsilon)^2} \rho$,
so that
\begin{equation*}
   |v| = |\Phi \ast \mathcal{N}(v)|\leq \Phi \ast |\mathcal{N}(v)| \lesssim  \frac{\epsilon}{(1+\epsilon)^2}\Phi \ast \rho.
\end{equation*}

We bound the convolution $\Phi\ast \rho$ by 
\begin{equation*}
 |\Phi \ast \rho| \lesssim \int \Phi(t)\frac{1}{1+|\alpha-t|^2}dt = \int_{|t-\alpha|\geq \frac{\alpha}{2}} \Phi(t)\frac{1}{1+|\alpha-t|^2}dt + \int_{|t-\alpha|< \frac{\alpha}{2}} \Phi(t)\frac{1}{1+|\alpha-t|^2}dt.
\end{equation*}
For the first integral on the right-hand side,
\begin{equation*}
   \int_{|t-\alpha|\geq \frac{\alpha}{2}}\Phi(t)\frac{1}{1+|\alpha-t|^2}dt \lesssim  \frac{1}{1+\alpha^2}\int\Phi(t) dt = \frac{1}{1+\alpha^2}.
\end{equation*}
For the second integral on the right-hand side, for large $\alpha$
\begin{equation*}
  \int_{|t-\alpha|< \frac{\alpha}{2}} \Phi(t)\frac{1}{1+|\alpha-t|^2}dt \lesssim  \frac{1}{\alpha^2} \int_{\frac{\alpha}{2}}^{\frac{3\alpha}{2}}\frac{1}{1+|\alpha-t|^2}dt \lesssim \pi \frac{1}{1+\alpha^2}.
\end{equation*}
Therefore, we show that $|v|\lesssim \frac{(1+\pi)\epsilon}{(1+\epsilon)^2}\jbx^{-2}.$

In order to give a more precise asymptotic characterization, we work in the weighted Sobolev spaces \begin{equation*}
H^{k,\sigma} = \{ u : (1+|\alpha|^2)^{\frac{\sigma}{2}}u \in H^k  \},
\end{equation*}
and prove the main result of this section, which is a rescaled version of Theorem \ref{TheoremThree}.

\begin{theorem}\label{t:rescaledAsymptoticExpansion}
    Let $\epsilon>0$ be small enough and there exists a solitary wave solution $\phi$ to~\eqref{e:babenkoeqnrescaled}.
    For any positive integer $N\geq 1$, there exists real numbers $a_j$, $1\le j\le N$ depending on $\epsilon$ and a function $g_N\in H^{2,2N}$ such that $\phi$ can be expressed as
    \begin{equation}
        \phi=\rho+\sum_{j=1}^{N} a_j\rho^j+g_N.\label{e:claimedform}
    \end{equation}

\end{theorem}

\begin{remark}
From \eqref{e:claimedform}, no odd powers of $\jbx^{-1}$ appear in the asymptotic expansion.
Additionally, the proof of the theorem goes through identically in the case that $\phi$ and $\alpha\phi$ lie in $H^k$ for some $k>2$, with the remainder term $g_{N}$ lying in $H^{k,2N}$.
\end{remark}

To prove this theorem, we will need a few lemmas about the operator $|D|$ and the function $\rho$.
The first result is that the set of polynomials in $\rho: = \frac{4}{1+\alpha^2}$ is closed under $|D|$:
\begin{lemma}\label{l:structureofmodDrho}
    For any integer $k\ge 1$, there exists a sequence of coefficients $r_j$, $1\le j\le k+1$ such that
    \begin{equation*}  |D|\rho^k=\sum_{j=1}^{k+1}r_j \rho^j.
    \end{equation*}
\end{lemma}

For the second result, we give a decomposition of $\rho^\ell \ast f$. 
\begin{lemma} \label{c:structureofrhofconvolutions}
     Let $f\in H^{k,2N}$ for some positive integers $k,N\ge 1$ be an even function, then there exists a sequence of coefficients $q_j$, $\ell \le j\le N$ and a function $g_N\in H^{\infty,2N}$ such that
    \begin{equation}
        \rho^\ell\ast f=
    \sum_{j=\ell}^{N}q_j\rho^j + g_N.\label{e:structureOfRhoAstf}
    \end{equation}
\end{lemma}

A direct consequence of above lemma where we choose $f = \rho^n$ is that   the set of polynomials in $\rho$ is closed under convolution, up to an error term controllable in an appropriate $H^{k,\sigma}$ space.
\begin{corollary}\label{l:structureofrhoconvolutions}
    Let $1\le m\le n$ be two positive integers, and let $N\ge m$ be a third positive integer.
    Then there exists a sequence of coefficients $p_j$, $m\le j\le N$ and a function $g_N\in H^{\infty,2N}$ such that
    \begin{equation}
        \rho^m\ast \rho^n=\sum_{j=m}^N p_j\rho^j + g_N.
    \end{equation}
\end{corollary}

The next result shows how the operator $|D|$ acts on the even functions of $H^{k,2N}$:
\begin{lemma}\label{l:structureofmodDf}
    Let $f\in H^{k,2N}$ for some positive integers $k,N\ge 1$ be an even function.
    Then there exists a sequence of coefficients $b_j$, $1\le j\le N$ and a function $g_N\in H^{k-1,2N}$ such that
    \begin{equation*}
        |D|f=\sum_{j=1}^N b_j\rho^j+g_N.
    \end{equation*}
\end{lemma}

The final lemma, in a similar spirit to the preceding lemma, shows that $(1+|D|)^{-1}$ preserves expressions of the form~\eqref{e:claimedform} up to raising the regularity of the remainder by one:

\begin{lemma}\label{l:structureOfInverseOperator}
    Let $f\in H^{k,2N}$ for some positive integers $k,N\ge 1$.
    Then there exists a sequence of coefficients $c_j$, $1\le j\le N$ and a function $g_N\in H^{k+1,2N}$ such that
    \begin{equation}
        (1+|D|)^{-1}f=\sum_{j=1}^N c_j\rho^j +g_N.\label{e:claimedFormInverse}
    \end{equation}
\end{lemma}

We prove Theorem \ref{t:rescaledAsymptoticExpansion} assuming above auxiliary results, and the proofs of these lemmas are deferred to the end of this section.

\begin{proof}[Proof of Theorem~\ref{t:rescaledAsymptoticExpansion}]
The asymptotic expansion is constructed inductively.
For our base case $N=1$, we have shown the pointwise bound  $|v(\alpha)| \lesssim \jbx^{-2}.$

\par We first obtain a pointwise bound for $\mathcal{N}(v)$.
We know $\rho$ is in $H^{k,1}$ for any $k$.
If $v$ is in $H^{k,1}$, then, since 
\begin{equation*}
    \alpha |D|v=|D|(\alpha v)+[\alpha,|D|]v=|D|(\alpha v)-Hv
\end{equation*} for all $v$ decaying at infinity, $|D|v$ is in $H^{k-1,1}$.
The product of two elements of $H^{k,\sigma}$ is in $H^{k,2\sigma}$, as long as $k\ge 1$.
Similarly the product of three elements of $H^{k,\sigma}$ is in $H^{k,3\sigma}$.

These facts and the explicit formulas~\eqref{e:nonlinearitygeneral} and~\eqref{e:nonlinearityexplicitdecay} tell us that for $v\in H^{2,1}$, we can write $\mathcal N(v)$ in the form
\begin{equation*}
    \mathcal N(v)=f_1+|D|f_2,
\end{equation*}
where $f_1\in H^{1,2}$ and $f_2\in H^{2,2}$ by grouping all terms with a $|D|$ in front into the $|D|f_2$ and all products where $|D|$ either doesn't appear or appears in the middle into $f_1$.

Lemma~\ref{l:structureofmodDrho} tells us that all the terms fully in terms of $\rho$ can be expressed as a power series in $\rho$, doing the truncation at order $1$ and placing the remaining terms in $H^{1,2}$.
Lemma~\ref{l:structureofmodDf} tells us that we can write \begin{equation*}
    |D|f_2=b_1 \rho +h_1
\end{equation*}
for some $h_1\in H^{1,2}$.

Given an expansion to order $N-1$ of the form~\eqref{e:claimedform}
\begin{equation*}
    v=\sum_{j=1}^{N-1}c_j\rho^j+g_{N-1},
\end{equation*}
the nonlinear expressions in~\eqref{e:nonlinearityexplicitdecay} are all one of
\begin{itemize}
    \item a product of powers of $\rho$,

    \item a product of $\rho$ with $g_{N-1}$, and thus in $H^{k,2N}$, or

    \item the product $g_{N-1}^2$, which is in $H^{k,2N}$,
\end{itemize}
or the derivative of such a term.

Applying Lemmas~\ref{l:structureofmodDrho} or~\ref{l:structureofmodDf} as appropriate, we get an expansion to order $N$ of the form~\eqref{e:claimedform}
\begin{equation*}
    \mathcal N(v)=\sum_{j=1}^{N}\tilde c_j\rho^j+\tilde g_{N}
\end{equation*}
for some $\tilde g_{N}\in H^{k-1,2N}$.

We now apply Lemma~\ref{l:structureOfInverseOperator} to this expansion of $\mathcal N(v)$ and get an expression of the form
\begin{equation*}
    v=\sum_{j=1}^{N}c_j\rho^j+g_{N}
\end{equation*}
as desired.
\end{proof}

\begin{proof}[Proof of Lemma~\ref{l:structureofmodDrho}]
    The function $\rho$ can be rewritten as
    \begin{equation*}
        \rho(\alpha)=\frac{1}{i}\left(\frac{2}{\alpha-i}-\frac{2}{\alpha+i}\right).   
    \end{equation*}
  As a consequence of fractional decomposition, for any positive integer $\ell$,  $\rho^\ell$ can be decomposed as the sum
    \begin{equation}
      \rho^{\ell}=\sum_{j=1}^{\ell}\beta^{(\ell)}_j (\alpha-i)^{-j}+\gamma^{(\ell)}_j (\alpha+i)^{-j},\label{e:partialfractionsrho}
    \end{equation}
    where $\beta^{(\ell)}_j$ and $\gamma^{(\ell)}_j$ are coefficients such that  $\beta^{(\ell)}_j=\gamma^{(\ell)}_j$ when $j$ is even and $\beta^{(\ell)}_j=-\gamma^{(\ell)}_j$ when $j$ is odd.
    Furthermore, we can write 
    \begin{equation*}
        \partial_\alpha\rho^k=\sum_{j=2}^{k+1}-j\beta^{(\ell)}_j (\alpha-i)^{-j}-j\gamma^{(\ell)}_j (\alpha+i)^{-j}.
    \end{equation*}
An important observation for the Hilbert transform is that 
\begin{equation}
        H\left((\alpha-i)^{-j}\right)=-i(\alpha-i)^{-j},\quad H\left((\alpha+i)^{-j}\right)=i(\alpha+i)^{-j}.\label{e:hilbertofinversepower}
    \end{equation}
   Applying~\eqref{e:hilbertofinversepower} and extracting powers of $\rho$ starting at $\rho^{k+1}$ using~\eqref{e:partialfractionsrho}, we can construct a sequence $r_j$ such that
    \begin{equation*}
        |D|\rho^k=\sum_{j=1}^{k+1}r_j\rho^j
    \end{equation*}
    as desired.
\end{proof}

In the following proofs, the error term $g_N$ can change from line to line as additional terms are placed into it, but it will always lie in the space claimed in the statement of the corresponding lemma.

We also remark that we will bound the $H^{k,\sigma}$ norm of a function $f$ by bounding either
\begin{equation*}
    \left\|\left(1-\partial_\alpha^2\right)^{\frac k 2}\left(1+|\alpha|^2\right)^{\frac\sigma 2}f\right\|_{L^2}
\end{equation*}
or
\begin{equation*}
    \left\|\left(1+|\alpha|^2\right)^{\frac\sigma 2}\left(1-\partial_\alpha^2\right)^{\frac k 2}f\right\|_{L^2}
\end{equation*}
depending on which is most convenient at a given point in the proof; both are equivalent to bounding the $H^{k,\sigma}$ norm because of the commutation relation between differentiation and multiplication.

\begin{proof}[Proof of Lemma~\ref{c:structureofrhofconvolutions}]
Taking a spatial dyadic decomposition \[f\approx\sum_{j=0}^\infty f_{j},\quad f_{0}=\chi (\alpha)f,\quad f_{j}=\left(\chi \left(\frac{\alpha}{2^j}\right)- \chi \left(\frac{\alpha}{2^{j-1}}\right)\right)f,\]
    where $\chi$ is a positive smooth even bump function that is adapted to $\{ |\alpha|\leq 2\}$ and equals $1$ on the ball $\{ |\alpha|\leq 1\}$. 
    Because $\alpha^{2N}f\in H^{k}$, we have
    \begin{equation*}
        \sum_{j=0}^\infty 2^{2Nj} \|f_{j}\|_{H^k}\approx\|f\|_{H^{k,2N}}<\infty.
    \end{equation*}
    
    We work first at the $L^2$ level and handle derivatives later.
    We evaluate the expression $\left(\rho^\ell\ast f\right)(\alpha)$ by looking at the contributions from $f_j$ for $2^j<|\alpha|$, $2^j\approx |\alpha|$, and $2^j> |\alpha|$ separately.
    In the following computations, the implicit constant in the symbol $\lesssim$ is always an absolute numerical constant not depending on $j$, $\alpha$, or $f_2$ and suppressed to make the computation clearer.

    \textbf{Case 1:} $2^j\ll |\alpha|$.
    On the support of $f_j$, we can bound $\alpha-\beta$ by a constant multiple of $\alpha$, so we have, also using H\"older's inequality,
    \begin{equation*}
        \left|\int \frac{ f_j(\beta)}{\left((\alpha-\beta)^2+1\right)^\ell}\,d\beta\right| \lesssim\frac{1}{(\alpha^2+1)^\ell}\int  f_j(\beta)\,d\beta\lesssim 2^{j/2}\rho^\ell\|f_j\|_{L^2}.
    \end{equation*}

    We can replace $\rho^\ell(\alpha-\beta)$ with powers of $(\alpha-\beta)^{-1}$ greater than or equal to $2\ell$.
    We write
    \begin{equation*}
        \frac{1}{\left((\alpha-\beta)^2+1\right)^\ell}=\frac{1}{(\alpha-\beta)^{2\ell}}+\frac{P_{2\ell-2}(\alpha-\beta)}{\left((\alpha-\beta)^2+1\right)^\ell(\alpha-\beta)^{2\ell}},
    \end{equation*}
    where $P_{2\ell-2}$ is a polynomial of degree $2\ell-2$, so that the difference is of order $-2\ell-2$.
    Repeating this process with the error term every step until the error is of order $-2N-2$ and thus can be safely added to the final remainder $g_N$, we can replaced convolution with $\rho^\ell$ by convolution with a series of even powers of $\alpha^{-1}$.

    We can write
    \begin{equation}
        \frac{1}{(\alpha-\beta)^{2m}}=\frac{1}{\alpha^{2m}}\frac{1}{\left(1-\frac \beta \alpha\right)^{2m}}=\sum_{n=2m}^{2N+2}\mu_n\frac{1}{\alpha^n}(-\beta)^{n-2m}+\frac{\mu_{2N+3}(-\beta)^{2N+3-2m}}{\alpha^{2N+3}\left(1-\frac{\beta}{\alpha}\right)^{2m}}\label{e:geometricSeries}
    \end{equation}
    for some coefficients $\mu_n$.
    Again the convolution with the final term can be placed in the overall error $g_N$.
    
    Because $f$ is assumed to be even, an odd power of $\beta$ integrated against $f_j(\beta)$ will give $0$, and only the even powers will remain.
    For even $n$, we can perform the same decomposition as above and write
    \begin{equation*}
        \int f_j(\beta)\mu_n\frac{1}{\alpha^n}(-\beta)^{n-2m} d\beta =\frac{\nu_{n,j}}{\alpha^n}=\sum_{p=\frac n 2}^N \nu_{p,n}\rho^p+\eta_{2N+2},
    \end{equation*}    
where the remainder $\eta_{2N+2}$ is of order $-2N-2$ and can thus be added to the overall error $g_N$.

    Performing this decomposition for each $n$ in~\eqref{e:geometricSeries}, and then doing that expansion for each $m$ between $\ell$ and $N$, inclusive, we get
    \begin{equation*}
        \rho^\ell \ast f_j = \sum_{m=\ell}^{N} \lambda_m \rho^m + g_N, \qquad g_N\in H^{k,2N}.
    \end{equation*}

    \textbf{Case 2:} $2^j\approx |\alpha|$.  We have, by Young's convolution inequality,
    \begin{equation*}
        \|\rho^\ell\ast  f_j\|_{L^\infty}\lesssim \|\rho^\ell\|_{L^2}\| f_j\|_{L^2}\lesssim \| f_j\|_{L^2}.
    \end{equation*}

    \textbf{Case 3:} $2^j\gg |\alpha|$.
    Here we have, bounding the $\rho^\ell$ term directly by $2^{-2j\ell}$ and using H\"older's inequality 
    \begin{align*}
        \left|\int \frac{ f_j(\beta)}{\left(1+(\alpha-\beta)^2\right)^\ell}d\beta\right|&\lesssim 2^{-2j\ell}2^{j/2}\| f_j\|_{L^2}.
    \end{align*}

    Summing these up over $j$, letting $m$ be the value of $j$ for which $2^{m-1}<|\alpha|<2^m$, we have a pointwise bound
    \begin{equation}
        |\left(\rho^\ell \ast  f\right)(\alpha)|\lesssim \rho^\ell(\alpha)\sum_{2^j<|\alpha|} 2^{j/2}\|\partial^i_\alpha f_j\|_{L^2}+\| f_{m}\|_{L^2}+\sum_{2^{j-1}\ge |\alpha|} 2^{j(1-4\ell)/2}\| f_j\|_{L^2}.
        \label{e:lemmaClaimIntermediate}
    \end{equation}

    We now need to show that for $j > \left\lceil\log_2 |\alpha|\right\rceil$, $\rho^\ell \ast f_j$ can be grouped into the error term $g_N$ lying in $H^{k,2N}$.

    For the case $2^j \approx |\alpha|$, we have the estimate 
    \begin{equation}
        \int \alpha^{4N} \left\| f_{m} \right\|_{L^2}^2\, d\alpha\lesssim \|f\|_{L^2}^2<\infty.\label{e:lemmaMiddleTerm}
    \end{equation}
    For the case $2^j \gg |\alpha|$, we have, bounding the integral in $\alpha$ by a sum and rearranging,
    \begin{align*}
        \int \alpha^{4N} \left(\sum_{2^{j-1}\ge|\alpha|}2^{j(1-4\ell)/2} \|f_j\|_{L^2}\right)^2\, d\alpha&\lesssim \sum_{m=0}^\infty 2^{m(1+4N)}\left(\sum_{j=m}^\infty 2^{j(1-4\ell)/2}\|f_j\|_{L^2}\right)^2\\
        &\lesssim \sum_{m=0}^\infty \left(\sum_{j=m}^\infty 2^{j(1-2\ell+2N)}\|f_j\|_{L^2}\right)^2\\
        &\lesssim \sum_{m=0}^\infty \left(\sum_{j=m}^\infty 2^{j(2-4\ell)}\right)\left(\sum_{j=m}^\infty 2^{2Nj}\|f_j\|_{L^2}\right)^2.
    \end{align*}
    The first inner sum is a convergent geometric series converging to a constant multiple of $2^{-m(4\ell-2)}$, and the second inner sum can be controlled by $\|f\|_{H^{0,2N}}$.
    This means that
    \begin{equation}
        \int \alpha^{4N} \left(\sum_{2^{j-1}\ge|\alpha|}2^{j(1-4\ell)/2} \|f_j\|_{L^2}\right)^2\, d\alpha\lesssim \|f\|_{H^{0,2N}}^2<\infty.\label{e:lemmaFinalTerm}
    \end{equation}
    Combining this with~\eqref{e:lemmaMiddleTerm} and applying Cauchy-Schwarz inequality, we see that the decomposition~\eqref{e:lemmaClaimIntermediate} gives the desired decomposition of $\rho^\ell\ast f$ from~\eqref{e:structureOfRhoAstf} at the $L^2$ level.

    We now prove the remainder term lies in $H^{2K,2N}$ for any positive integer $K$.
    By interpolation, this ensures that $g_N\in H^{\infty,2N}$ as desired.

    A computation shows that 
    \begin{equation}
        \partial^{2K}_\alpha\rho^\ell=\sum_{j=\ell+K}^{\ell+2K} p_j\rho^j\label{e:structureOfPartial2krho}
    \end{equation}
    for some real coefficients $p_j$.
    Writing
    \begin{equation*}
        g_N=\left(\rho^\ell\ast f-\sum_{j=\ell}^N q_j\rho^j\right)
    \end{equation*}
    and differentiating $2K$ times, placing all $2K$ derivatives on the $\rho^\ell$ term in the convolution, we have
    \begin{equation*}
        \partial_\alpha^{2K}g_N=\left(\sum_{j=\ell+K}^{\ell+2K} p_j\rho^j\ast f - \sum_{j=\ell}^{N}\sum_{n=K}^{2K}q_{j,n}\rho^{j+n}\right),
    \end{equation*}
    where $q_{j,n}$ are real coefficients arising from applying~\eqref{e:structureOfPartial2krho} to $q_j\rho^{j}$ in the expansion up to order $N$ of $g_N$.
    
    Repeating the argument above shows that the remainder when subtracting this new power series in $\rho$ from the sum of the convolutions with powers of $\rho$ must lie in $H^{0,2N}$.
    This means that $g_N\in H^{\infty,2N}$ as desired.    
\end{proof}

\begin{proof}[Proof of Lemma~\ref{l:structureofmodDf}]
    By applying Lemma~\ref{l:structureofmodDrho} inductively, we know that for each integer $i\ge 0$, $|D|^i\rho$ is a polynomial in $\rho$ of degree $i+1$.
    This means that for any integer $j\ge 1$, we can find a sequence of coefficients $p_i$ such that 
    \begin{equation}
        \rho^j=\sum_{i=0}^{j-1}p_i |D|^i \rho.\label{e:structureOfRhoJ}
    \end{equation}

    Using the fact that the Fourier transform of $\rho$ equals $\sqrt{\frac{\pi}{2}}e^{-|\xi|}$,~\eqref{e:structureOfRhoJ} tells us that the Fourier transform of a polynomial in $\rho$ is given by a sum of terms of the form $|\xi|^ie^{-|\xi|}$:
    \begin{equation*}
        \mathcal F\left(\sum_{i=1}^j p_i \rho^i\right)=
        \sum_{i=0}^{j-1} q_i |\xi|^i e^{-|\xi|}.
    \end{equation*}

Homogeneity considerations tell us that, heuristically, $|D|$ should act like convolution with $\alpha^{-2}$.
    To make this more precise and give a full asymptotic expansion, we re-express its symbol $|\xi|$ in a careful way to obtain a more detailed understanding of the (mildly singular) low-frequency behavior.
    
    Using the Taylor expansions at the origin of the terms $|\xi|^i e^{-|\xi|}$, we can find coefficients $c_j$ depending on $N$ and a $C^{2N}$ function $R_{2N}$ such that

    \begin{equation}
        |\xi|=c_0+\sum_{j=1}^{2N} c_j|\xi|^{j-1}e^{-|\xi|}+R_{2N}(\xi).\label{e:structureOfModXi}
    \end{equation}
    Taking the inverse Fourier transform, we see that we can represent $|D|$ as a sum of convolution with a multiple of the $\delta$ distribution, convolution with a polynomial in $\rho$, and convolution with a remainder term $\check R_{2N}(\alpha)$.
    
    Because $R_{2N}(\xi)$ grows like $|\xi|$ asymptotically, convolution with $\check R_{2N}$ loses exactly one derivative.
    Because $R_{2N}(\xi)$ is $2N$-times continuously differentiable in $\xi$, its kernel decays strictly faster than $\alpha^{2N}$ and thus convolution with $\check R_{2N}$ preserves spatial decay of order $\langle \alpha \rangle^{-2N}$.
    This means that $\check R_{2N}$ sends $H^{k,2N}$ to $H^{k-1,2N}$.
    
    Since convolution with $\delta$ is the identity, it remains to analyze the effect of convolution with a polynomial in $\rho$.
    Convolution with a smooth function does not lose any derivatives, so we only need to consider how convolution with a power of $\rho$ affects spatial decay.

   Given any function $f\in H^{k,2N}$, we can construct the expansion~\eqref{e:structureOfModXi} of $|\xi|$ to order $2N$ and extract the coefficient on $\rho^j$ in the final decomposition that comes from~\eqref{e:structureOfRhoAstf} at each order $j$.
    
    The final remainder $g_N$ is constructed from the $\delta$ and $\check R_{2N}$ terms, the errors from each of the powers $\rho^j$, and the terms of the form $\rho^j$ for $j>N$.
    Considering the spaces these contributions lie in, we see that $g_N\in H^{k-1,2N}$ as claimed.
\end{proof} 

\begin{proof}[Proof of Lemma~\ref{l:structureOfInverseOperator}]
    As in the proof of the preceding lemma, we approximate the symbol  $(1+|\xi|)^{-1}$ by functions of the form $|\xi|^{j-1}e^{-|\xi|}$ to handle the (mildly singular) behavior at the origin, where the symbol is not smooth,
    \begin{equation}
        \frac{1}{1+|\xi|}=\sum_{j=0}^{2N-1} (-1)^j|\xi|^j+\frac{|\xi|^{2N}}{1+|\xi|}.\label{e:lemma54firstExpansion}
    \end{equation}

    Taking Taylor expansions around the origin, we can write
    \begin{equation}
        \sum_{j=0}^{2N-1} (-1)^j|\xi|^j=\sum_{j=1}^{2N}c_j|\xi|^{j-1}e^{-|\xi|}+\tilde R_{2N}(\xi),\label{e:lemma54secondExpansion}
    \end{equation}
    where $\tilde R_{2N}(\xi)$ is $2N$ times continuously differentiable.
    Note that unlike in the previous proof, the constant coefficient $c_0$ is zero.
    
    The initial error term $\frac{|\xi|^{2N}}{1+|\xi|}$ is also $2N$ times continuously differentiable since the factors of $\xi$ in the numerator absorb the $\delta^{j}$-type singularities from derivatives falling on the denominator.
    Moreover, all the $|\xi|^{j-1}e^{-|\xi|}$ terms decay at infinity faster than any polynomial, so combining the expansions~\eqref{e:lemma54firstExpansion} and~\eqref{e:lemma54secondExpansion}, we have a sequence of coefficients $c_j$ such that
    \begin{equation*}
        \frac{1}{1+|\xi|}=\sum_{j=1}^{2N}c_j|\xi|^{j-1}e^{-|\xi|}+R_{2N}(\xi),
    \end{equation*}
    where $R_{2N}$ is a $C^{2N}$ function decaying like $|\xi|^{-1}$ for large $\xi$.

    By repeatedly applying Lemma~\ref{l:structureofmodDrho} and taking the inverse Fourier transform, we see that we can represent $(1+|D|)^{-1}$ as a sum of convolution with a polynomial in $\rho$ and convolution with a remainder term $\check{R}_{2N}(\alpha)$.

    Convolution with a polynomial in $\rho$ gives an expansion of the desired form~\eqref{e:claimedFormInverse} with an error term $g_N\in H^{k+1,2N}$ by Lemma~\ref{l:structureofrhoconvolutions}.

    The remainder term $R_{2N}(\xi)$ is in $C^{2N}$ and decays like $|\xi|^{-1}$, so convolution with $\check R_{2N}$ preserves spatial decay and gains one derivative in $L^2$, so that $\check R_{2N}$ is smoothing of order $1$ and $\check R_{2N}(\alpha)\ast f\in H^{k+1,2N}$. 

    Summing the contributions to the error from the convolutions with powers of $\rho$ (which is in $H^{\infty,2N}$) and the error from convolution with $\check R_{2N}$, we have $g_N\in H^{k+1,2N}$ as desired.
\end{proof}

\section{Continuation of solitary waves for velocity} \label{s:perturbation}

In this section, we discuss the continuation problem to solitary waves with higher speed.
By Theorem~\ref{t:TheoremTwo} we know that for velocities close to $-\frac g \gamma$, a solitary wave exists, and the construction guarantees that it depends continuously on the velocity. 
We now ask to what extent this family of solitary waves can be extended continuously as a function of the velocity.

We first investigate the local problem, finding conditions under which the existence of a solitary wave at some velocity $c_1$ in the above range implies the existence of a nearby solitary wave at a close velocity $c_2$. 
Using a continuity argument, the local result then 
implies a sufficient continuation criteria for the 
solitary waves family, and the failure of the local result gives some possible features of the solitary wave of maximal speed (if it exists).
We conclude with the proof of Theorem~\ref{t:maximal} characterizing the behavior near or at the endpoint velocity.

To set the stage for our results, we start
with the Babenko equation at a velocity $c_1$ 
as above, as in \eqref{e:BabenkoEqnLR}.
\begin{equation*}
    (g+ c_1\gamma-c_1^2|D|)U_1 = - \frac{\gamma^2}{2}U_1^2-gU_1|D|U_1 -\frac{g}{2}|D|U_1^2 +\frac{\gamma^2}{2}\left(U_1|D|U_1^2 - U_1^2|D|U_1 - \frac{1}{3}|D|U_1^3\right).
\end{equation*}
Given a  solution $U_1$ of velocity $c_1$, a key role is played by the associated linearized operator $\cL_{U_1}$, defined by
\begin{equation}
    \begin{aligned}
    \cL_{U_1}w&=\left(g+c_1\gamma+\gamma^2U_1\right)w-c_1^2|D|w\\
    &+g\left(U_1 |D| w +  |D| U_1 w+|D|(U_1 w)\right)\\
    &-\frac{\gamma^2}{2} \left(2U_1 |D|(U_1 w)+w|D|U_1^2-U_1^2|D|w-2U_1w|D|U_1-|D|(U_1^2w)\right),
    \end{aligned}\label{e:linearizedExtension}
\end{equation}
which will be used in a fixed point argument to find a solitary wave solution close to $U_1$ with velocity $c_2$.

The operator $\mathcal L_{U_1}$ is self-adjoint in $L^2$ because $|D|$ is self-adjoint and the coefficients given in terms of $U_1$ and its fractional derivatives are all real.

Using the paradifferential calculus as in Section \ref{s:ellipticity}, we can distribute $|D|$ to each factor like $\partial_\alpha$.
The third line of \eqref{e:linearizedExtension} has commutator structures
and will not contribute to the top-order behavior of $\cL_{U_1}$, so that $\cL_{U_1}$ is elliptic as long as $\sup_\alpha U_1 < \frac{c_1^2}{2g}$.
For $c_1$ close enough to $-\frac{g}{\gamma}$, 
\begin{equation*}
    \sup_\alpha U_1\lesssim |g+c_1\gamma|\ll \frac{c_1^2}{2g},
\end{equation*} so that $\cL_{U_1}$ is elliptic for velocities close to the critical velocity.

Moreover, because the solitary wave $U_1$ decays at infinity and $\cL_{U_1}$ is elliptic, its continuous spectrum is determined purely by the constant-coefficient portion $g+c_1\gamma-c_1^2|D|$,
so that the continuous spectrum of $\cL_{U_1}$ is 
$(-\infty,g+c_1\gamma]$.

To analyze the discrete spectrum of $\cL_{U_1}$, we perform the rescalings from Section~\ref{s:slowexist}, writing
\begin{equation*}
    U_1=\frac{g+c_1\gamma}{\gamma^2}\phi\left(\frac{(g+c_1\gamma)\alpha}{c_1^2}\right),
\end{equation*}
and using the fixed-point construction of Section~\ref{s:slowexist} to write
\begin{equation}
    U_1=\frac{g+c_1\gamma}{\gamma^2}\rho\left(\frac{(g+c_1\gamma)\alpha}{c_1^2}\right)+O_{L^\infty}\left((g+c_1\gamma)^2\right),
\end{equation}
we can view $\cL_{U_1}$ as a perturbation of a rescaled version of $-1-|D|+\rho$, where the perturbation coming from the difference between $U_1$ and (rescaled) $\rho$ and the higher-order terms in $\cL_{U_1}$ are all $O_{L^\infty}\left((g+c_1\gamma)^2\right)$.

We can now use the upper semicontinuity of the spectrum to guarantee that for $c_1$ close enough to $-\frac g \gamma$, the discrete spectrum remains close to that of $-1-|D|+\rho$.

The discrete spectrum of $-1-|D|+\rho$ has only one zero eigenvalue, with corresponding eigenvector $\rho'$.
This eigenvalue can be associated to the translation symmetry of the equation, so by working in the class of even functions (which is preserved under $|D|$ and multiplication, and which contains $U_1$), we can guarantee that the discrete spectrum of $\cL_{U_1}$ is bounded away from $0$ and thus that $\cL_{U_1}$ is invertible. 

\begin{proposition}\label{p:localExtension}
Suppose that for some fixed nonzero vorticity $\gamma$ we have an even solitary wave solution $U_1$ of velocity $c_1$ in the weighted Sobolev
space $X$. 
Assume in addition that the self-adjoint operator $\cL_{U_1}$ arising from linearizing~\eqref{e:BabenkoEqnLR} around the solution $U_1$
 has no zero eigenvalues
(when restricted to even functions).
Then for the velocity $c_2$ in a small neighborhood of $c_1$, we also have (an even) solitary wave in the same space, depending smoothly on the velocity $c_2$.

The size $\epsilon'$ of the interval around $c_1$ of possible velocities depends on the norm of the initial solitary wave profile $\|U_1\|_X$,
the distance between the maximal height of $U_1$ and the maximal height $\frac{c_1^2}{2g}$, and the smallest absolute value $|\lambda_m|$ of an eigenvalue (other than the trivial zero eigenvalue with corresponding eigenvector $V_1$) of $\mathcal L_{U_1}$. 
\end{proposition}

\begin{proof}
We seek a second solution $U_2$ with velocity $c_2$ close to $c_1$, which would solve
\begin{equation*}
    (g+ c_2\gamma-c_2^2|D|)U_2 = - \frac{\gamma^2}{2}U_2^2-gU_2|D|U_2 -\frac{g}{2}|D|U_2^2 +\frac{\gamma^2}{2}\left(U_2|D|U_2^2 - U_2^2|D|U_2 - \frac{1}{3}|D|U_2^3\right).
\end{equation*}

Letting $w=U_2-U_1$, and subtracting the equation satisfied by $U_1$ from the equation satisfied by $U_2$, we get the difference equation
\begin{align}
    \cL_{U_1}w&=-(c_2-c_1)(\gamma-(c_2+c_1)|D|)(U_1+w)+\mathcal N(w,U_1),\label{e:extensionEqn}
    \end{align}
where $\mathcal N(w,U_1)$ represents the terms with at least two factors of $w$ and arising from the quadratic and cubic parts of the difference of the two Babenko equations for $U_1$ and $U_2$.

Assuming the hypotheses of the proposition, $\mathcal L_{U_1}^{-1}$ is well-defined, so the solution of \eqref{e:linearizedExtension} is a fixed point of
\begin{equation*}
    P_{U_1}(w)=\mathcal L_{U_1}^{-1}\left(-(c_2-c_1)(\gamma-(c_2+c_1)|D|)(U_1+w)+\mathcal N(w,U_1)\right).
\end{equation*}
We can perform a fixed point argument as in Section~\ref{s:slowexist} for small $\epsilon'=|c_2-c_1|$ and $w$ in a small ball of radius $\delta'=k\epsilon'$ around the origin in the weighted Sobolev space $X$ given by~\eqref{e:fixedpointspace} as long as neither of the following happen:
\begin{enumerate}
    \item The operator $\mathcal L_{U_1}$ becomes non-elliptic.
    This occurs if and only if 
    \begin{equation*}
        \sup\limits_\alpha U_1\ge \frac{c^2}{2g}
    \end{equation*}
    by Proposition~\ref{p:maximumheight}.

    \item The operator $\mathcal L_{U_1}$ develops another zero eigenvalue.
    The zero eigenvalue of the initial operator $\mathcal L = (-1-|D|+\rho)$ corresponds to the translational symmetry of the problem, so restricting to even functions enables the argument to go through.
    A new zero eigenvalue of $\mathcal L_{U_1}$ would require a more detailed understanding of the spectrum of $\mathcal L_{U_1}$ to handle.
\end{enumerate}

As long as neither of these two conditions are satisfied, we can extend the range of allowed velocities beyond the range given by Theorem~\ref{t:TheoremOno}.
We can construct solitary waves with velocities lying in an open interval 
$\left(c_M,-\frac{g}{\gamma}\right)$ or $\left(-\frac g \gamma,c_M\right)$
where $|c_M|$ is the supremum of possible speeds and the order of the endpoints depends on the sign of $\gamma$.

\textit{A-priori}, we do not know if the sum of the lengths of the intervals of continuation $\epsilon'_j$ from running the fixed point argument around new solitary waves $U_j$ is finite, but we can at least consider the factors affecting the length of the interval of extension, which will shrink as the operator norm of $P_{U_1}$ as an operator from $X$ to $X$ grows.
We unpack the constants affecting this operator norm.

The operator norm of $P_{U_1}$ contains a term proportional to the reciprocal of the smallest (nontranslational) eigenvalue $\lambda_m$, which appears in estimating the norm of $\cL_{U_1}^{-1}$ as an operator from $L^2$ to $L^2$.
It also contains a term proportional to the constant in the elliptic regularity estimate for $\cL_{U_1}^{-1}$, as well as additional terms depending on the size of $U_1$ which appear when we upgrade the $L^2$ estimates to higher regularity.
As in Section~\ref{s:ellipticity}, we can write $\cL_{U_1}$ paradifferentially, and, using the estimates of Lemmas~\ref{l:highHigh}-\ref{l:paraproductDifference}, view the difference
\begin{equation*}
    L_{U_1}^0\cL_{U_1}-(g+c_1\gamma-c_1^2|D|+2gU|D|)
\end{equation*}
as an order $0$ operator satisfying the estimate
\begin{equation*}
    \|L_{U_1}^0w\|_{H^s}\lesssim \|U_1\|_{H^2}\|w\|_{H^s}
\end{equation*}
for all $s\ge 0$.
The constant in the elliptic regularity estimate for $(g+c_1\gamma-c_1^2|D|+2gU|D|)U$ is proportional to
\begin{equation*}
    \left(\frac{c_1^2}{2g}-\sup_{\alpha}U_1(\alpha)\right)^{-1}
\end{equation*}
As in~\eqref{e:ellipticRegularityStart}-\eqref{e:ellipticRegularityEnd}, with $(L_{U_1}+\gamma^2 U_1)(u-v)$ replacing $\rho (u-v)$, the operator norm of $P_{U_1}$ picks up terms proportional to $\|U_1\|_{H^2}$ and the elliptic regularity constant $\left(\frac{c_1^2}{2g}-\sup_{\alpha}U_1(\alpha)\right)^{-1}$.

Collecting the various terms that influence how small $\epsilon'$ and $\delta'$ can be, we see that as the velocity approaches the maximal velocity $c_M$, if neither the maximal height condition (1) nor the spectral condition (2) are being approached, we could also have
\begin{enumerate}\setcounter{enumi}{2}
    \item The solitary wave $U$ becomes too large in the weighted Sobolev space $X$.
\end{enumerate}

As long as none of the maximal height condition for non-ellipticity (1), the spectral condition for non-invertibility (2), or the norm blowup condition (3) occur, we can run the fixed point argument and extend the range of velocities to a small open interval around $c_1$.
\end{proof}

\begin{remark}
    It is possible that condition (3) above can be relaxed to a smaller set of norms, possibly just the $L^\infty$ norms of $U_1$ and perhaps $|D|U_1$, but we do not pursue a more detailed and quantitative study of the extension problem here.
\end{remark}

\begin{proof}[Proof of Theorem~\ref{t:maximal}]
    The initial construction from Theorem~\ref{t:TheoremOno} and the local continuation result of Proposition~\ref{p:localExtension} guarantees the existence of an open interval of velocities $\left(c_M,-\frac g \gamma\right)$ for $\gamma>0$ or $\left(-\frac g \gamma,c_M\right)$ for $\gamma<0$ for which a solitary wave exists.
    Moreover, the fixed point construction shows that the map $c\mapsto U_c$ from a possible velocity $c$ to the corresponding solitary wave profile $U_c$ is continuous  from $c$ to $X$.
    
    If the interval for which solitary waves exist is finite, then this requires the length $\epsilon'_c$ of the interval of extension to be decreasing, for if $\epsilon'_c$ was bounded below then the velocity could be extended to $\pm \infty$.
    This means that if $c_M>-\infty$ for $\gamma>0$ (or $c_M<\infty$ for $\gamma<0$), there must be a sequence   $\{c_n\}$ approaching $c_M$ at which $\epsilon'_{c_n}$ is approaching $0$.
    
    By the previous proposition, we know that $\epsilon'_{c_n}$ is inversely proportional to all of \begin{enumerate}
        \item the norm of the solitary wave profile $U_n$ in $X$;

        \item the distance between the height $\sup_\alpha U_n$ of $U_n$ and the maximum height $\frac{c_n^2}{2g}$; and

        \item the absolute value of the smallest eigenvalue of $\cL_{U_n}$.
    \end{enumerate}
    This means that if the interval of possible velocities is finite, one of the conditions (2)-(4) of Theorem~\ref{t:maximal} must occur.
\end{proof}

\bibliographystyle{plain}
\bibliography{refs}
\end{document}